\documentclass[a4paper,10pt]{article}
\usepackage[utf8]{inputenc}

\usepackage{amsmath}%
\usepackage{amsfonts}%
\usepackage{amssymb}%
\usepackage{amsthm}
\usepackage{graphicx}
\usepackage{subfig}
\usepackage{tikz}

\usepackage[abs]{overpic}
\newtheorem{theorem}{Theorem}[section]

\newtheorem{lemma}[theorem]{Lemma}

\newtheorem{proposition}[theorem]{Proposition}

\theoremstyle{definition}
\newtheorem{definition}[theorem]{Definition}

\theoremstyle{definition}

\newtheorem{example}[theorem]{Example}

\newcommand{\ad}{\mbox{ad}} 
\newcommand{\im}{\mbox{im}}
\newcommand{\vspan}{\mbox{span}}
\newcommand{\ii}{{\text{i}}}
\newcommand{\R}{\mathbb{R}}

\allowdisplaybreaks
\title{Unfolding of nilpotent equilibria of degree 4 in Hamiltonian systems with 2 degrees of freedom}
\author{Giannis Moutsinas
\\[6pt]
Mathematics Institute, University of Warwick\\[6pt]
{\small
E-mail: Giannis.Moutsinas@gmail.com}}

\begin{document}

\maketitle

\begin{abstract}
 We consider Hamiltonian systems of two degrees of freedome having a nilpotent equilibrium point with only one eigenvector.
 We provide the universal unfolding of such equilibrium, provided a non-degeneracy condition holds.
 We show that the only co-dimension 1 bifurcations that happen in the unfolding are of two types:
 the normally hyperbolic or elliptic centre-saddle bifurcations and the supercritical Hamiltonian-Hopf bifurcation.
\end{abstract}

\tableofcontents

\section{Introduction}
\label{ch:intro}

One of the few possible methods to study the dynamics of Hamiltonian systems is to focus on neighbourhoods of their equilibria.
In the case where the eigenvalues of the linearized system have non-vanishing real parts, the motion is completely determined by the linearization.
In the case of imaginary eigenvalues,  the presence of resonances affects dramatically the dynamics.

In the present paper we will focus on equilibria with vanishing eigenvalues.
The case with degree 2 nilpotent matrix arises when modeling optics in an $\chi^2$-medium and was studied in \cite{Wagen02}.
A degree 3 nilpotent matrix cannot appear in two degrees of freedom and a degree 1 nilpotent matrix is the zero matrix, which has co-dimension 10, see \cite{galin82}.

We will consider a system of two degrees of freedom with an equilibrium, such that the linearized system has a degree 4 nilpotent matrix.
A fist step towards the study of the unfolding was performed in \cite{hanssmann06}. In the present paper we complete the study.

\paragraph{Structure of the paper} \hspace{0em}

Section \ref{ch:intro} reiterates known results that are used in the study. In section \ref{ch:results} the non-linear unfolding of the system is studied. We will see that even though the linear unfolding is of co-dimension 2, the non-linear unfolding requires 3 parameters. In section \ref{ch:truncated-hamiltonian} we will see that under a non-degeneracy condition, the system can be simplified by truncating certain 3rd order terms. Finally in section \ref{ch:param_redux} we will see how the parameters of the truncated system can be reduced from 3 to 2.

\subsection{Hamiltonian systems on $\R^4$}

Here we review briefly some basic facts from the theory of Hamiltonian systems.
Due to the scope of the paper we concentrate on $\R^4$.
For a thorough treatment of the theory see \cite{arnold90}.

We can view $\R^4$ as a symplectic manifold by defining the symplectic form $$ \omega = d q_1 \wedge d p_1 + d q_2 \wedge d p_2. $$
Here we use the canonical coordinates, i.e. a point on $\R^4$ is represented as $(q_1,q_2,p_1,p_2)$.
The symplectic form can equivalently be defined as
$$ \omega(u,v) = u^\intercal \Omega v, $$
where $\Omega$ is the $2m \times 2m$ matrix $\left( \begin{smallmatrix} 0&E_m\\ -E_m&0 \end{smallmatrix} \right)$, with $E_m$ the $m\times m$ identity matrix.
Then the Hamiltonian vector field can be written as
$$ X_H = \Omega\, \nabla H, $$
with $\nabla H$ the divergence vector of $H$.

A change of coordinates is called symplectic if it preserves the symplectic form.
In particular, on $(\R^4,\omega)$ a linear transformation is symplectic if and only if its matrix $P$ satisfies
$$ P^\intercal\,\Omega\, P = \Omega. $$

\subsection{Linear normal form}

A linear Hamiltonian system can be defined by a Hamiltonian function which is a polynomial of degree 2.
Naturally we can assume that the equilibrium is located at the origin, so we can restrict our class of Hamiltonian functions to homogeneous polynomials of degree 2.
Such Hamiltonian, $H$, defines a linear system with the matrix $A=\Omega\, \mathbf{H}_H$,where $ \mathbf{H}_H$ is the Hessian matrix of $H$.

Using symplectic transformations we can define equivalence classes of Hamiltonian systems.
Then we can choose one representative for each class.
We call this representative \textit{normal form} of the class.
The following proposition gives one such choice for nilpotent equilibria.

\begin{proposition}[\cite{williamson36}]
If $A$ is a $2m\times 2m$ real symmetric matrix whose Jordan form is a Jordan block of dimension $2m$ and eigenvalue zero, then there exists a symplectic matrix $P$ such that $A\,P=P\,A_0$ and $A_0$ corresponds to the Hamiltonian function
\begin{equation}
H_0(x)=\frac{1}{2} x^\intercal A_0 x= \pm \frac{1}{2}\left( \sum_{i=1}^{m-1}p_i p_{m-i}-\sum_{i=1}^{m}q_i q_{m+1-i} \right)-\sum_{i=1}^{m-1}p_i q_{i+1}.
\label{eq:williamson-norm-general}
\end{equation}
\end{proposition}

In the case of 2 degrees of freedom the Hamiltonian function \eqref{eq:williamson-norm-general} becomes
\begin{equation}
H_1(x)=\frac{1}{2}p_1^2-q_1 q_2-p_1 q_2.
\label{eq:will-norm-form-2d}
\end{equation}
An equivalent Hamiltonian function, used in \cite{cushman-sanders86} , is
\begin{equation}
H_0(x)=\frac{p_2^2}{2}-p_1 q_2.
\label{eq:lin-norm-form}
\end{equation}
From now on, the Hamiltonian \eqref{eq:will-norm-form-2d} will be refered to as \textit{Williamson normal form} and the Hamiltonian \eqref{eq:lin-norm-form} will be refered to as \textit{standard normal form}.

Let $H_1(x)=\frac{1}{2} x^\intercal A_1 x$ and $H_0(x)=\frac{1}{2} x^\intercal A_0 x$. Then a symplectic matrix $P$ that changes $H_1(x)$ to $H_0(x)$ has to satisfy $P\, \Omega \, A_0=\Omega \, A_1 \, P$ and $P^\intercal\,\Omega\, P=\Omega$. One such matrix is
\begin{equation}
P=
\left(
\begin{array}{cccc}
 0 & 1 & 0 & -1 \\
 0 & 0 & 1 & 0 \\
 0 & 0 & 0 & 1 \\
 -1 & 0 & 0 & 0
\end{array}
\right).
\label{eq:matrix_P}
\end{equation}
From here on, if the linear part of a Hamiltonian system in 2 degrees of freedom is nilpotent of degree 4, then it will be assumed to be in the standard normal form.

\subsection{Versal deformations of linear systems}
\label{ch:versal-def}

Given a linear system, the Jordan form of its matrix provides all the information needed to 
solve it.
However, if the system depends on parameters, the transform to Jordan form may be discontinuous with respect to the parameters.
In order to define an appropriate normal form we need some additional notions, \cite{arnold71}.


\begin{definition}
Let $\mu \in \mathbb{R}^n$ and $A(\mu) \in \mathbb{R}^{m^2}$ be a matrix whose elements are formal power series in $\mu$. If $A(0)=A_0$, then we say that $A(\mu)$ is \textit{a deformation of the matrix $A_0$}.
\end{definition}
 
\begin{definition}
A deformation $A(\mu)$ of $A_0$ is called \textit{versal}\footnote
{
The name \textit{versal} is obtained by the word \textit{universal} discarding the prefix \textit{uni} indicating uniqueness.
},
if for any other deformation $B(\nu)$ of $A_0$, there exists a map $\phi : \nu \rightarrow \mu$ smooth in a neighbourhood of the point $\nu=0$ and a symplectic matrix $S(\nu)$ depending smoothly on $\nu$, such that
\begin{equation*}
\begin{array}{c}
\phi(0)=0, \; S(0)=E, \\
B(\nu)=S(\nu)\, A(\phi(\nu))\, S^{-1}(\nu).
\end{array}
\end{equation*}
Moreover, a versa deformation is called \textit{universal} if the change of parameters $\phi$ is uniquely determined by the matrix $B(\nu)$.
\end{definition}

A versal deformation is the most general deformation there can be for a given matrix $A_0$ in the sense that it can be transformed into any other deformation.

%

Here we are interested in the deformations of Hamiltonian matrices.
A versal deformation for every Hamiltonian matrix in Williamson normal form is given in \cite{galin82}.

\paragraph{Versal deformation of nilpotent systems in 2 degrees of freedom} \hspace{0em}

It is proven in \cite{galin82} that a versal deformation of a linear nilpotent system in 2 degrees of freedom in Williamson normal form \eqref{eq:will-norm-form-2d} is the Hamiltonian
\begin{equation}
H_w(x)=\frac{1}{2}p_1^2-q_1 q_2-p_1 q_2+\mu_1 p_1 p_2+\nu_1 \frac{p_2^2}{2}.
\label{eq:galin-norm-form}
\end{equation}
This shows that the co-dimension of this linear system is 2.

\begin{figure}[p]
\centering
\includegraphics[width=\textwidth]{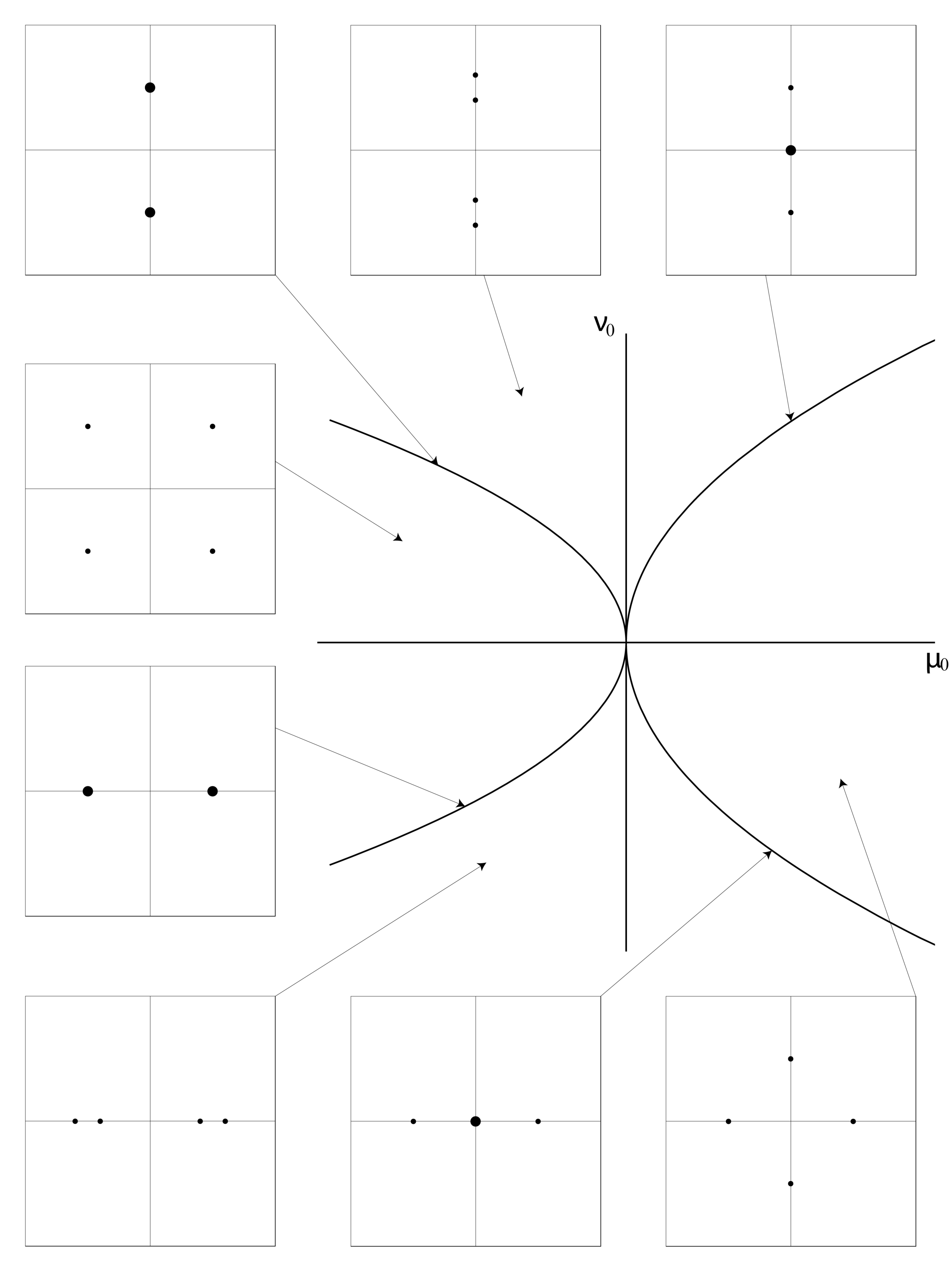}
\caption{Eigenvalue configurations of the Hamiltonian \eqref{eq:vers-deform-normal-form}.}
\label{fig:lin-eig-conf}
\end{figure}

\begin{proposition}
The Hamiltonian function
\begin{equation}
H_0(x)=\frac{p_2^2}{2}-p_1 q_2+\mu_0 \frac{q_1^2}{2}+\nu_0\left( \frac{q_2^2}{2}+\frac{3}{4}p_2 q_1\right)
\label{eq:vers-deform-normal-form}
\end{equation}
is a versal deformation of a linear nilpotent Hamiltonian system in the standard normal form \eqref{eq:lin-norm-form}.
\end{proposition}

\begin{proof}
With a symplectic transformation by the matrix $P$ in \eqref{eq:matrix_P} the Hamiltonian \eqref{eq:galin-norm-form} becomes
\begin{equation*}
H_1(x)=\frac{p_2^2}{2}-p_1 q_2+\nu_1 \frac{q_1^2}{2}+\mu_1 p_2 q_1
\end{equation*}
and therefore a versal deformation of a nilpotent system at the standard normal form.

Let $J_0(\mu_0,\nu_0)$ and $J_1(\mu_1,\nu_1)$ be the Hamiltonian matrices of the Hamiltonian functions $H_0(x)$ and $H_1(x)$, respectively. Since $H_1(x)$ is a versal deformation, it is sufficient to show that there exists a map $\phi : (\mu_1,\nu_1) \rightarrow (\mu_0,\nu_0)$, smooth in a neighbourhood of the point $(0,0)$, and a symplectic matrix $S(\mu_1,\nu_1)$ depending smoothly on $(\mu_1,\nu_1)$, such that
\begin{equation*}
\begin{array}{c}
\phi(0,0)=(0,0), \; S(0,0)=E, \\
J_1(\mu_1,\nu_1)=S((\mu_1,\nu_1))\, J_0(\phi((\mu_1,\nu_1)))\, S^{-1}((\mu_1,\nu_1)).
\end{array}
\end{equation*}
The eigenvalues of $J_0(\mu_0,\nu_0)$ are $\pm \frac{1}{2} \sqrt{-5 \nu_0\pm 4 \sqrt{\mu_0+\nu_0^2}}$ and the eigenvalues of $J_1(\mu_1,\nu_1)$ are $\pm\sqrt{-\mu_1 \pm\sqrt{\nu_1}}$. From this we find 
\begin{equation*}
(\mu_0,\nu_0)=\phi(\mu_1,\nu_1)=(\nu_1-\frac{16}{25}\mu_1^2,-\frac{4}{5}\mu_1).
\end{equation*}
Then we search for a matrix that satisfies both
\begin{equation*}
J_1(\mu_1,\nu_1)S((\mu_1,\nu_1))=S((\mu_1,\nu_1)) J_0(\phi((\mu_1,\nu_1)))\\
\end{equation*}
and
\begin{equation*}
S^\intercal((\mu_1,\nu_1))\,\Omega\, S((\mu_1,\nu_1))=\Omega.
\end{equation*}
Recall that $\Omega$ is the $4 \times 4$ matrix $\left( \begin{smallmatrix} 0&E\\ -E&0 \end{smallmatrix} \right)$ and $E$ the $2 \times 2$ identity matrix.
One such matrix is
\begin{equation*}
S(\mu_1,\nu_1)=\left(
\begin{array}{cccc}
 1 & 0 & 0 & 0 \\
 0 & 1 & 0 & 0 \\
 0 & -\frac{2}{5}\mu_1 & 1 & 0 \\
 -\frac{2}{5}\mu_1 & 0 & 0 & 1
\end{array}
\right).
\end{equation*}
\end{proof}

The eigenvalue configurations of the Hamiltonian \eqref{eq:vers-deform-normal-form} is given in Figure \ref{fig:lin-eig-conf}. Zero eigenvalues occur when $\mu_0=\frac{9}{16}\nu_0^2$ and two pairs of double eigenvalues occur when $\mu_0=-\nu_0^2$.

\subsection{Non-linear normal form}
\label{ch:norm-form-Ham-systems}

Similarly to the linear case, we can define equivalence classes in the set of Hamiltonian systems using canonical transformations and we can choose one representative of each class.
This representative is called the \textit{normal form}.
We can construct the required canonical transformations by using the flow of chosen Hamiltonian systems.

Let $M$ be a symplectic manifold.
A Hamiltonian system on $M$ defines an one-parameter flow on $M$ and by fixing the parameter it can be viewed as map from $M$ to $M$.
The key observation is that this defines a canonical transformation.
We will describe briefly the procedure in this section.
For a more detailed description see \cite{arnold06}.

\begin{definition}
Let $(M,\omega)$ be a symplectic $2m$-dimensional manifold, $f,g \in C^\infty(M)$ and $X_f$, $X_g$ the Hamiltonian vector fields of $f$ and $g$ respectively. The bilinear map $\{\,,\,\}:C^\infty(M)\times C^\infty(M) \rightarrow C^\infty(M)$ defined by 
\begin{equation*}
\{f,g\}=\omega(X_f,X_g)
\end{equation*}
is called the \textit{Poisson bracket} of $f$ and $g$. In local canonical coordinates the Poisson bracket takes the form
\begin{equation*}
\{f,g\}=\sum_{i=1}^m \partial_{q_i}f \,\partial_{p_i}g-\partial_{p_i}f \,\partial_{q_i}g.
\end{equation*}
\end{definition}

\begin{definition}
 By fixing $f$ we get the linear map $\{f,\,\}:C^\infty(M) \rightarrow C^\infty(M)$. This is called the \textit{adjoint map} of $f$ $$\ad_f(\,):=\{f,\,\}.$$
\end{definition}

Let $ x=(q_1,\dots,q_m,p_1,\dots,p_m)^\intercal$.
Abusing notation, let us denote the function $x\mapsto q_i$ by $q_i$. Similarly for $p_i$.
Then Hamilton's equations can be written as
$$\dot{x}=-\ad_H (x).$$
The formal solution of this equation is
$$ x(t)=\phi^t_H(x_0):=\exp(-t\,\ad_H)(x_0).$$
The exponential defined by
\begin{equation*}
\exp(-t\,\ad_f)(\,):=\sum_{n=0}^\infty \frac{(-1)^n }{n!}\,t^n\ad_f^n(\,)
\end{equation*}
and has the property $\exp(\ad_f)(\{g,h\})=\{\exp(\ad_f)(g),\exp(\ad_f)(h)\},$
and by this it can be shown that the transformation $x\mapsto \exp(\ad_f)(x)$ defines a (formal) canonical transformation, see \cite{basov10}.

\subsubsection{Normal forms near equilibria}

In order to study the dynamics in a neighbourhood of an equilibrium we recall that any neighbourhood on $M$ is diffeomorphic to a neighbourhood in $\mathbb{R}^{2m}$.
So without loss of generality we can study the corresponding Hamiltonian system on a neighbourhood of the origin in $\mathbb{R}^{2m}$.

Let $\mathcal{P}_n$ be the vector space of homogeneous polynomials of degree $n$ on $\mathbb{R}^{2m}$. Then for $f \in \mathcal{P}_n$ the adjoint map $\ad_f:\mathcal{P}_m \rightarrow \mathcal{P}_{m+n-2}$ maps homogeneous polynomials to homogeneous polynomials. In particular when $f \in \mathcal{P}_2$, it holds that $\ad_f:\mathcal{P}_m \rightarrow \mathcal{P}_m$.

Let $H$ be a Hamiltonian function without constant term.
Since the origin is an equilibrium, it has no first order terms either, so we have
\begin{equation*}
H=\sum_{k=2}^m H_k+R_m,
\end{equation*}
with $H_k \in \mathcal{P}_k$ and $R_m$ satisfying $R_m(0)=D R_m(0)=\dots=D^{m} R_m(0)=0$.

We define the canonical transformation $\phi_f:=\phi^1_f=\exp(-\ad_f)$, with $f\in C^\infty(\mathbb{R}^{2m})$, and we have
\begin{equation*}
H\circ\phi_f=\sum_{k=2}^\nu H_k+\{H_2,f\}+\sum_{k=3}^\nu\{H_k,f\}+\dots,
\end{equation*}
see \cite{hanssmann06}.

This shows that at the term $H_k$ any element of the image of $\ad_{H_2}$ can be added.
The key observation is that if $f \in \mathcal{P}_n$, then any $\{H_k,f\}$ with $k>2$ gives a homogeneous polynomial of degree greater than $n$.
Then $f$ can be chosen to be such that $\{H_2,f\}$ gives the desired terms and $H_n$ can be normalized without having any effect on the lower order terms.
This means that $f$ can be chosen initially to be in $\mathcal{P}_3$ to normalize $H_3$ without producing terms in $\mathcal{P}_2$, then it can be chosen to be in $\mathcal{P}_4$ to normalize $H_4$ without producing terms in $\mathcal{P}_2$ and $\mathcal{P}_3$.
One may continue and inductively normalize all $H_n$ up to any power.

We see that since we can freely add elements of $\im\,\ad_{H_2}$ to $H_k$ by choosing $f$, we can choose to transform it to $\hat{H}_k$ such that $\vspan(\hat{H}_k)\cap \im\,\ad_{H_2}=\{0\}$, with $\vspan(\hat{H}_k)$ being the linear subspace of $\mathcal{P}_k$ spanned by the monomials in $H_k$.
In this way $H_2$ determines the normal form of $H$.

A standard result from representation theory states that if $H_2=S+N$ is the Jordan decomposition of $H_2$, then $\ad_{H_2}=\ad_S+\ad_N$ is the Jordan decomposition of $\ad_{H_2}$, see \cite{humphreys78}. So in order to check whether $\ad_{H_2}$ is semi-simple it is sufficient to check whether the linearized system at the equilibrium has a semi-simple matrix.

If $\ad_{H_2}:\mathcal{P}_k \rightarrow \mathcal{P}_k$ is semi-simple, its eigenvectors span the whole space and it can be diagonalized over $\mathbb{C}$. We immediately see that $\mathcal{P}_k=\im\,\ad_{H_2}\oplus\ker\,\ad_{H_2}$, which holds  also over $\mathbb{R}$.

On the other hand, if the linear part of $H$ is $H_2=S+N$, with $S,N\in \mathcal{P}_2$, $N\ne0$, such that $\ad_{H_2}=\ad_S+\ad_N$ is the Jordan decomposition of $\ad_{H_2}$ and $\ad_N$ is nilpotent.
Then the system can be normalized further.

It was shown in \cite{cushman-sanders86} that $N$ can be embedded in a subalgebra of $\ker\,\ad_S$ that is isomorphic to $\mathfrak{sl}(2,\mathbb{R})$, i.e. there are elements $M,T\in \ker\,\ad_S \cap \mathcal{P}_2$ such that
\begin{equation*}
\{N,M\}=T, \;\; \{N,T\}=2N, \;\; \{M,T\}=-2M.
\end{equation*}
By this the splitting
\begin{equation*}
(\ker\,\ad_S \cap \ker\,\ad_M) \oplus (\ker\,\ad_S\cap\im\,\ad_N) =\ker\,\ad_S
\end{equation*}
is derived,
which ensures that the normalization can be done in two steps. Initially, the system is normalized with respect to $S$ and then with respect to $N$ without undoing the achievements of the first step, see \cite{vdMeer82}.

\subsection{Bifurcations of equilibria}

We will briefly discuss the two bifurcations of equilibria that will appear in the present analysis.

\subsubsection{Centre-saddle bifurcation}

Intuitively, the centre-saddle bifurcation happens when two equilibria collide and disappear.
In order for this to happen, the eigenvalues of the two equilibria have to converge and at least one eigenvalue must vanish.

In Hamiltonian systems, since if $\lambda$ is an eigenvalue then $-\lambda$, $\overline{\lambda}$ and $-\overline{\lambda}$ are also eigenvalues, the eigenvalues vanish always in pairs.
The simplest case of this bifurcation is in a one degree of freedom system, as shown in Example \ref{ex:centre-saddle-bif}.
This bifurcation is called \textit{centre-saddle bifurcation} because at this simple case it involves a center and a saddle.
For reasons that will be apparent later, this bifurcation is also called \textit{fold bifurcation}.

\begin{example}
Consider the Hamiltonian $H=\frac{1}{2}p^2-\frac{1}{3}q^3+\mu q$.
Its equations of motion are
\begin{equation*}
\begin{split}
\dot{q}&=p,\\
\dot{p}&=q^2-\mu.
\end{split}
\end{equation*}
If $\mu>0$, the equilibria are $(q_0,p_0)=(\pm \sqrt{\mu},0)$.
The equilibrium $(\sqrt{\mu},0)$ is a saddle and the equilibrium $(-\sqrt{\mu},0)$ is a centre.
When $\mu=0$, there is only one equilibrium $(0,0)$ with a $2\times2$ Jordan block of zero eigenvalue.
Finally, when $\mu<0$ the system has no equilibria.
\label{ex:centre-saddle-bif}
\end{example}

If the non-vanishing eigenvalues are real, then it was shown in \cite{Carr06} that there always exists a manifold, called the \textit{centre manifold}, on which the system takes the form of the Example \ref{ex:centre-saddle-bif}.
If the non-vanishing eigenvalues are imaginary, then it was shown in \cite{MR1227332} that the same can be done in integrable and near integrable systems.

However, the linear part of the system can always be transformed to the form of two separable linear systems, one nilpotent and another either hyperbolic or elliptic.
Then a centre-saddle bifurcation happens as long as the coefficient of the term $q^3$ does not vanish.

\subsubsection{Hamiltonian-Hopf bifurcation}
\label{sect:ham-hopf-bif}

The simplest generic case of the Hamiltonian-Hopf bifurcation happens in 2 degrees of freedom.
The system has to have a double eigenvalue at the bifurcation.
There are two possibilities of a system having double non semi-simple eigenvalues:
\begin{itemize}
 \item A system with four imaginary eigenvalues that meet on the imaginary axis and split into four complex eigenvalues.
 \item A system with four real eigenvalues that meet on the real axis and split again into four complex.
\end{itemize}
The second case is not viewed as a bifurcation because the equilibrium is hyperbolic throughout the process.
A system of the first form is given in Example \ref{ex:hamiltonian-hopf-bif-lin}.

\begin{example}
Consider the Hamiltonian system on $\mathbb{R}^4$ given by the Hamiltonian $H=q_1 p_2-q_2 p_1+\frac{1}{2}(q_1^2+q_2^2)+\frac{\mu}{2}(p_1^2+p_2^2)$. This gives the following equations of motion
\begin{align*}
\dot{q_1}&=-q_2+\mu p_1, \\
\dot{q_2}&=q_1+\mu p_2, \\
\dot{p_1}&= -q_1-p_2, \\
\dot{p_2}&= -q_2+p_1.
\end{align*}
The equilibrium is the point $(0,0)$ with eigenvalues $\pm\sqrt{-1 -\mu \pm 2 \sqrt{\mu } }$. If $\mu<0$ the system has four complex eigenvalues.
As $\mu$ goes to zero the eigenvalues tend to $\pm \ii$ and then they split to pairs $\pm \lambda_1 \ii$ and $\pm \lambda_2 \ii$ as $\mu$ becomes positive.
\label{ex:hamiltonian-hopf-bif-lin}
\end{example}

A Hamiltonian system with linear part as in Example \ref{ex:hamiltonian-hopf-bif-lin} generically undergoes the Hamiltonian-Hopf bifurcation.
In general the linear part of the normal form of such systems is
$$H_2=\omega S+\alpha N+\mu M,$$
where
\begin{align*}
 S&=q_1 p_2-q_2 p_1, \\
 N&=\frac{1}{2}\big(q_1^2+q_1^2\big), \\
 M&=\frac{1}{2}\big(p_1^2+p_1^2\big).
\end{align*}
Since $S$ corresponds to a semi-simple Hamiltonian matrix, it is called the semi-simple part of the Hamiltonian and $N$ is called the nilpotent part.

The normal form of such a system generically is
\begin{equation*}
\begin{split}
 H &=\omega S+\alpha N+\mu M+b M^2 \\
    &\phantom{=}+c S M+d S^2+\dots .
\end{split}
\end{equation*}
Under the condition that $\alpha$ and $b$ are non-zero, the terms in the second line do not influence qualitatively the system.

Depending on the sign of $\alpha b$, there are two types of the bifurcation.
If $\alpha b$ is positive the bifurcation is called supercritical and almost all orbits of the system are quasi-periodic.
This is called \textit{soft loss of stability}, since the orbits are bounded before and after the bifurcation even if the stability of the equilibrium changes.
If $\alpha b$ is negative the bifurcation is called subcritical and the majority of the orbits leave any neighbourhood of the origin.
This is called \textit{hard loss of stability}.
See \cite{vdMeer82,vdMeer85,vdMeer86}.

\section{Nilpotent equilibrium}
\label{ch:results}

Let $H$ be a Hamiltonian function on $\mathbb{R}^4$ and let the linear part be nilpotent of degree 4.
Then there is a linear symplectic transformation that will transform the linear part into
\begin{equation*}
H_2=\frac{p_2^2}{2}-p_1 q_2.
\end{equation*}
It is shown in \cite{cushman-sanders86} that the monomials in the kernel of $\ad\,H_2$ are generated by
\begin{align*}
& q_1,\;\, \frac{q_2^2}{2}+\frac{3}{4}p_2 q_1,\;\,\frac{3 p_1 q_1^2}{2}+\frac{3 p_2 q_1 q_2}{2}+\frac{2 q_2^3}{3} \\
&\text{and } \frac{3}{4}p_1^2 q_1^2+\frac{3}{2}p_1 p_2 q_1 q_2-\frac{1}{2}p_2^3 q_1+\frac{2}{3}p_1 q_2^3-\frac{1}{4}p_2^2 q_2^2.
\end{align*}
Thus a Hamiltonian with $H_2$ as linear part in normal form will be of the form
\begin{equation}
\begin{split}
 H_{\kappa,\mu,\nu} =& \frac{p_2^2}{2}-p_1 q_2 +\kappa q_1 +\mu \frac{ q_1^2}{2}+\nu\left( \frac{q_2^2}{2}+\frac{3}{4}p_2 q_1 \right)+a_1 \frac{ q_1^3}{6} \\
&+a_2\, q_1 \left( \frac{q_2^2}{2}+\frac{3}{4}p_2 q_1 \right) +a_3 \left( \frac{3 p_1 q_1^2}{2}+\frac{3 p_2 q_1 q_2}{2}+\frac{2 q_2^3}{3} \right)+\dots
\end{split}
\label{eq:nilpot_Ham_general}
\end{equation}

Notice that $\kappa$, $\mu$ and $\nu$ can change the nilpotence of the equilibrium, but $a_1$, $a_2$ and $a_3$ cannot.
For this reason $a_1$, $a_2$ and $a_3$ will be considered to be parameters and $a_1$, $a_2$ and $a_3$ will be considered to be constants.
Comparing the above Hamiltonian with the versal deformation of the linear system, we notice that we have an extra parameter $\kappa$.
For this reason one can expect that the parameters can be reduced to two.
However there is no obvious way in which this can be done at the present stage.
Later in the article we will reconcile the linear with the non-linear theory.

\begin{theorem}
If $a_1 \ne 0$ then for a small enough neighbourhood of the origin in $(q_1,q_2,p_1,p_2)$ and a small enough neighbourhood of the origin in $(\kappa,\mu,\nu)$,  the versal unfolding of the nilpotent equilibrium is given by surfaces diffeomorphic to the surfaces depicted in Figure \ref{fig:hopf-fold-real}.
On the red surface the Hamiltonian-Hopf bifurcation happens and it is always of the supercritical type.
On the blue surface the fold bifurcation happens.
On the black line where the two surfaces meet, the equilibrium is nilpotent.
\end{theorem}

Notice that even though the Hamiltonian system has 3 parameters only bifurcations of co-dimension 1 appear in the unfolding.

At the origin the two surfaces have a 3rd order tangency.
This implies that as the neighbourhood in the space of parameters becomes small, the space between them shrinks fast.
The Hamiltonian-Hopf bifurcation happens in a system at 1:-1 resonance and the elliptic fold bifurcation happens in a system at 1:0 resonance.
This implies that in the space between the two surfaces every resonance appears.

This section and the next one are dedicated to the proof of the theorem.
From now on we will assume that the coefficient $a_1$ does not vanish.

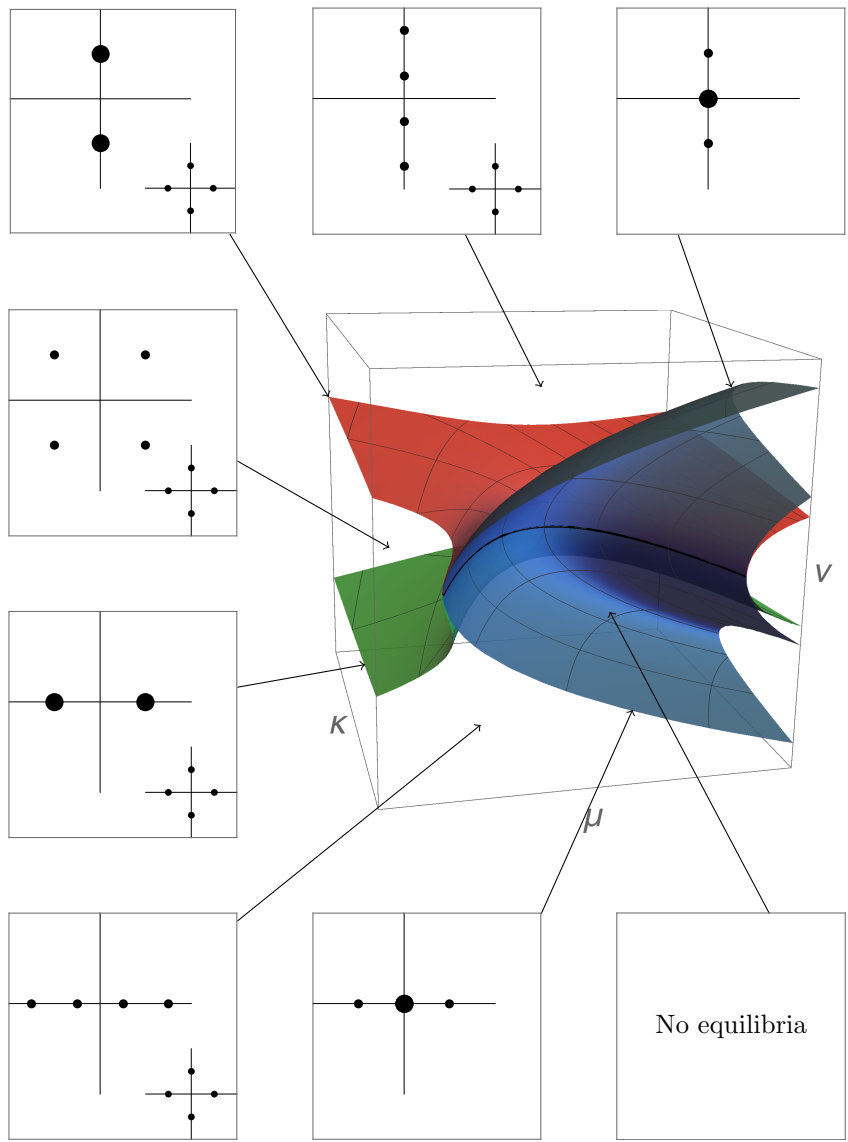
\begin{figure}[p]
\begin{tikzpicture}
 \node at (6,6) {\includegraphics[width=200pt]{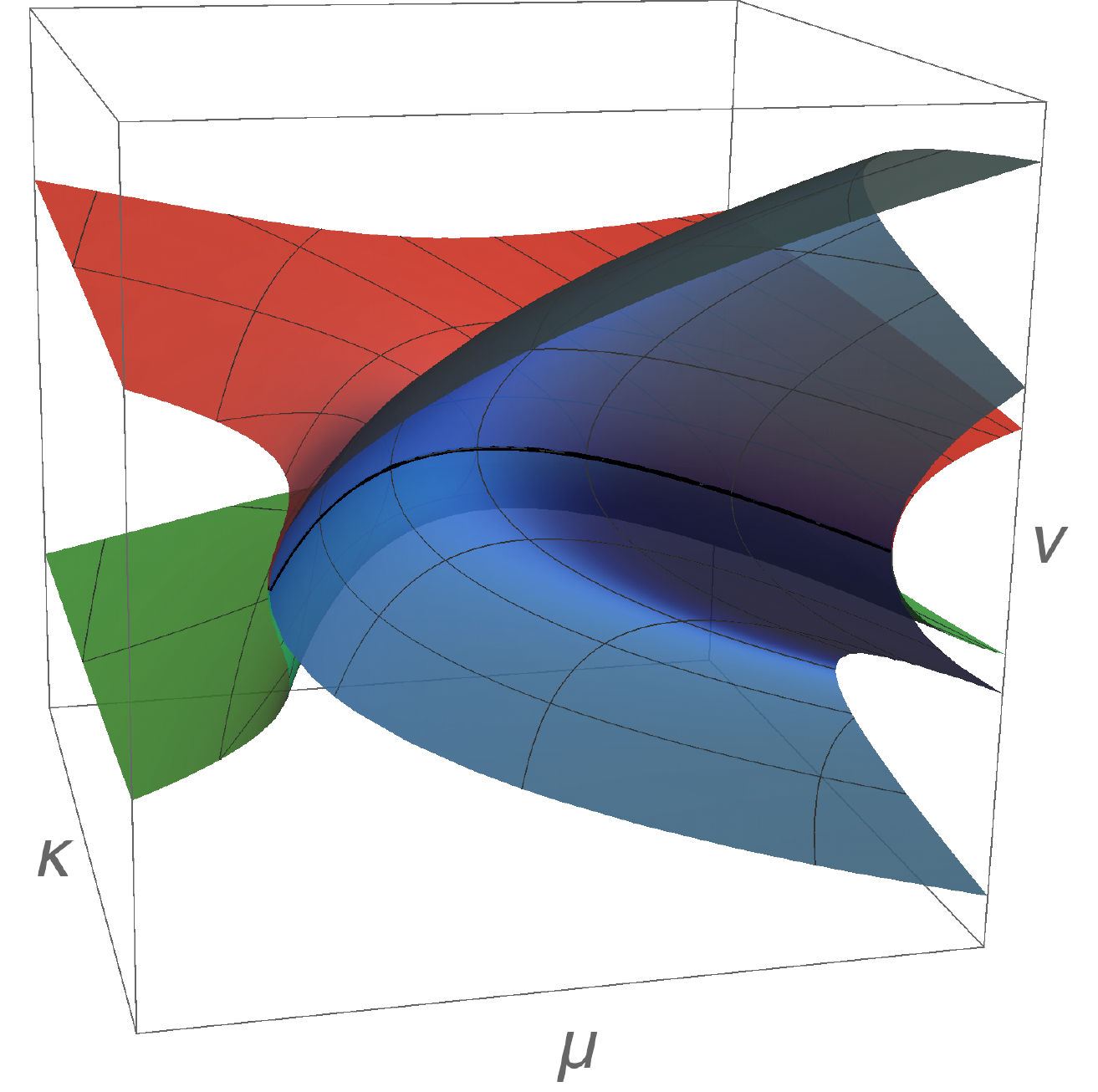}};
 \draw[->] (7.3,10.5) -- (8,8.48);
 \draw[->] (4.5,10.5) -- (5.5,8.48);
 \draw[->] (1.4,10.51) -- (2.7,8.35);
 \draw[->] (1.5,7.5) -- (3.5,6.35);
 \draw[->] (1.5,4.5) -- (3.18,4.8);
 \draw[->] (1.5,1.4) -- (4.7,4);
 \draw[->] (5.5,1.5) -- (6.7,4.2);
 \draw[->] (8.5,1.5) -- (6.4,5.5);
 \node at (0,12) {\includegraphics[width=85pt]{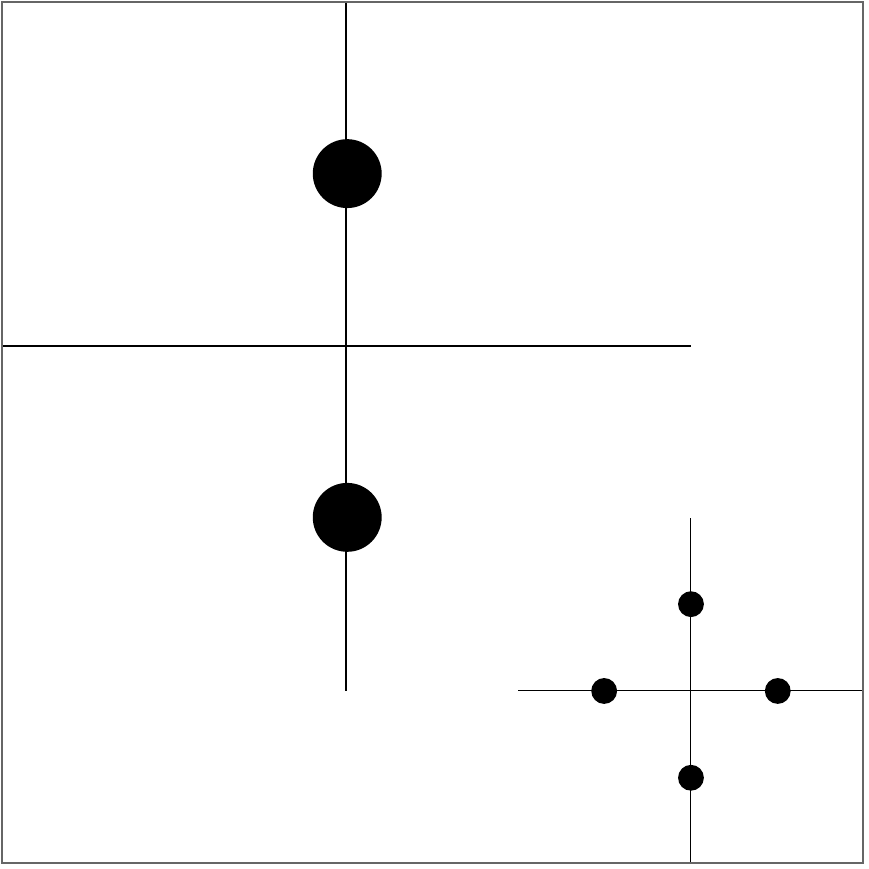}};
 \node at (0,8) {\includegraphics[width=86pt]{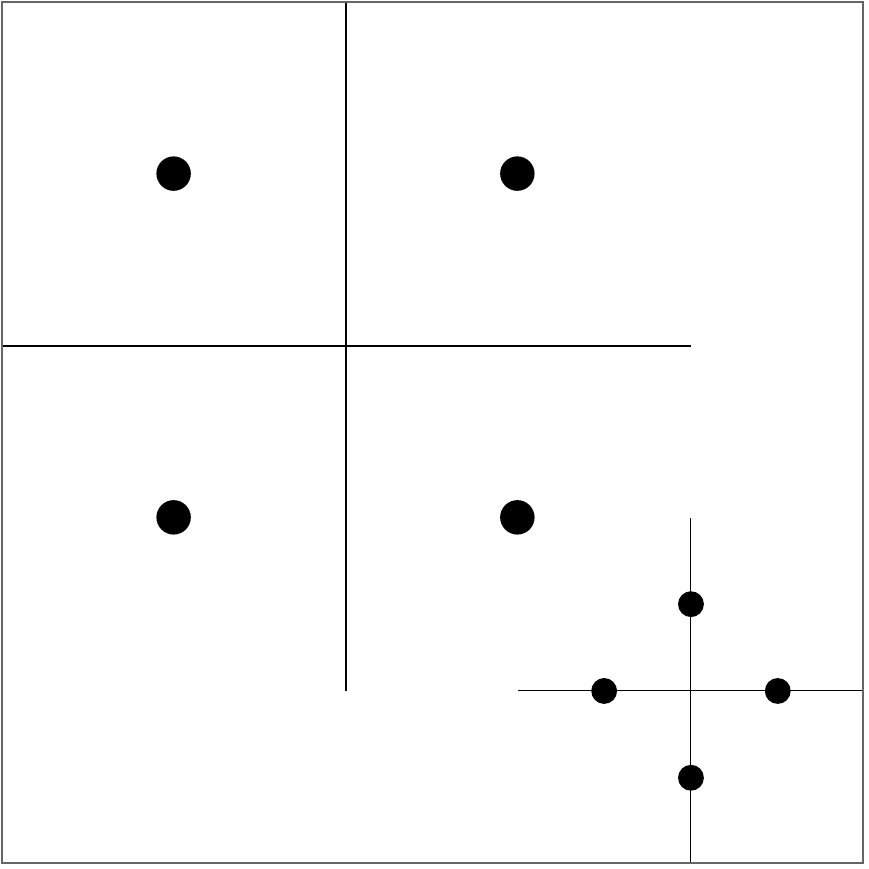}};
 \node at (0,4) {\includegraphics[width=86pt]{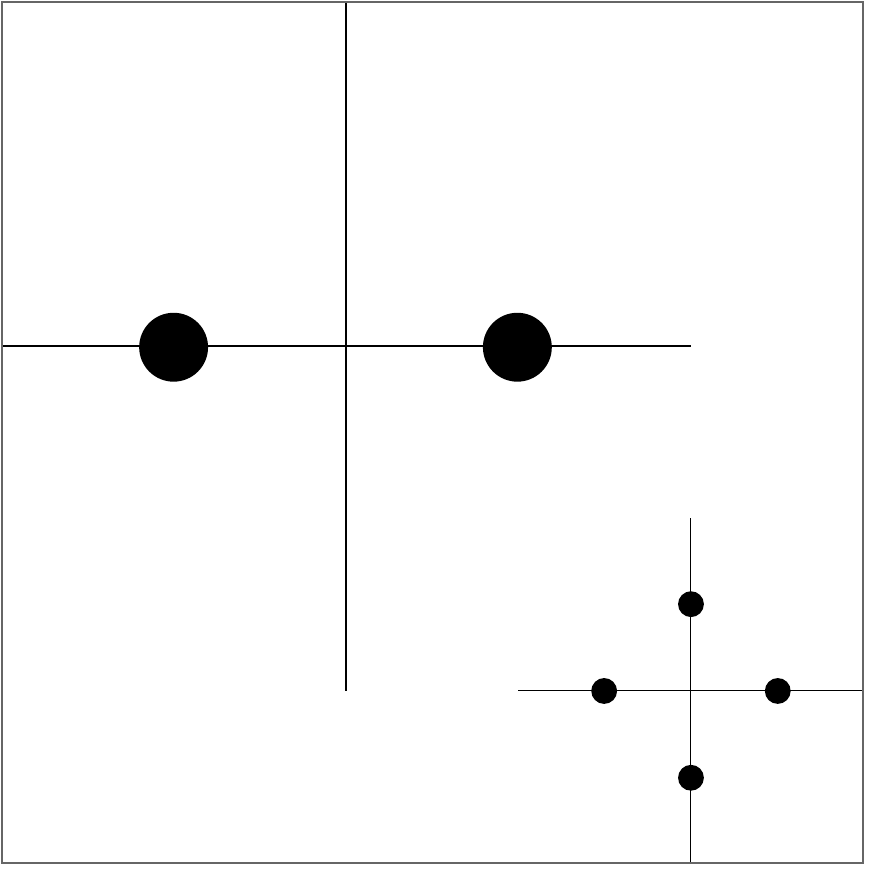}};
 \node at (0,0) {\includegraphics[width=86pt]{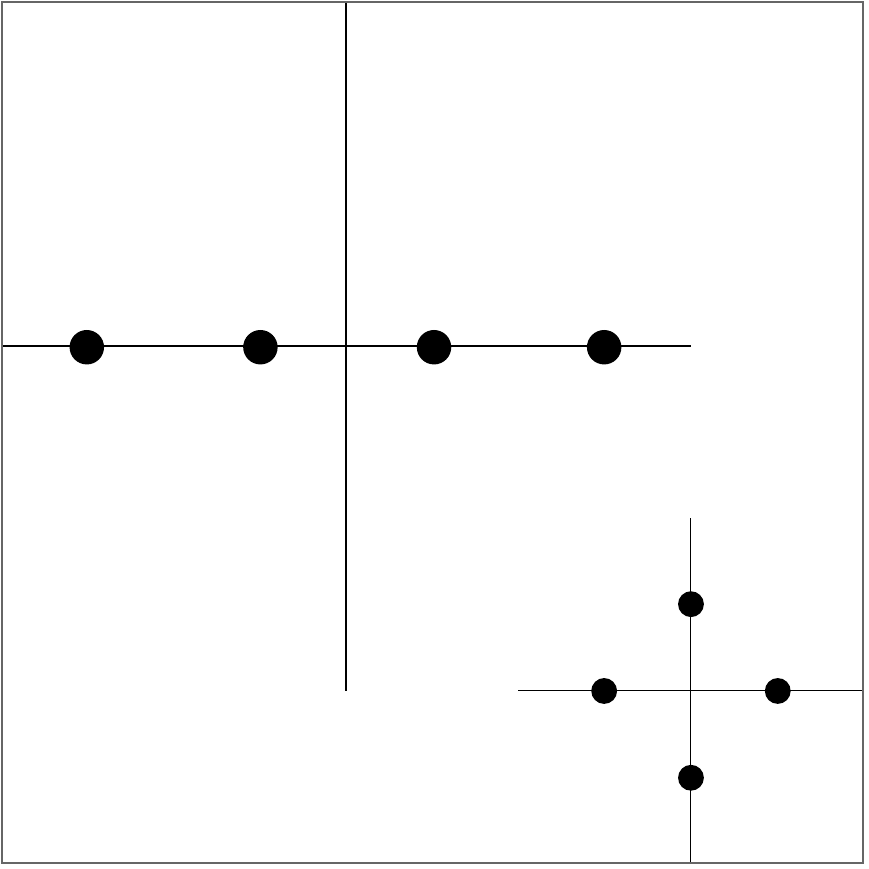}};
 \node at (4,0) {\includegraphics[width=86pt]{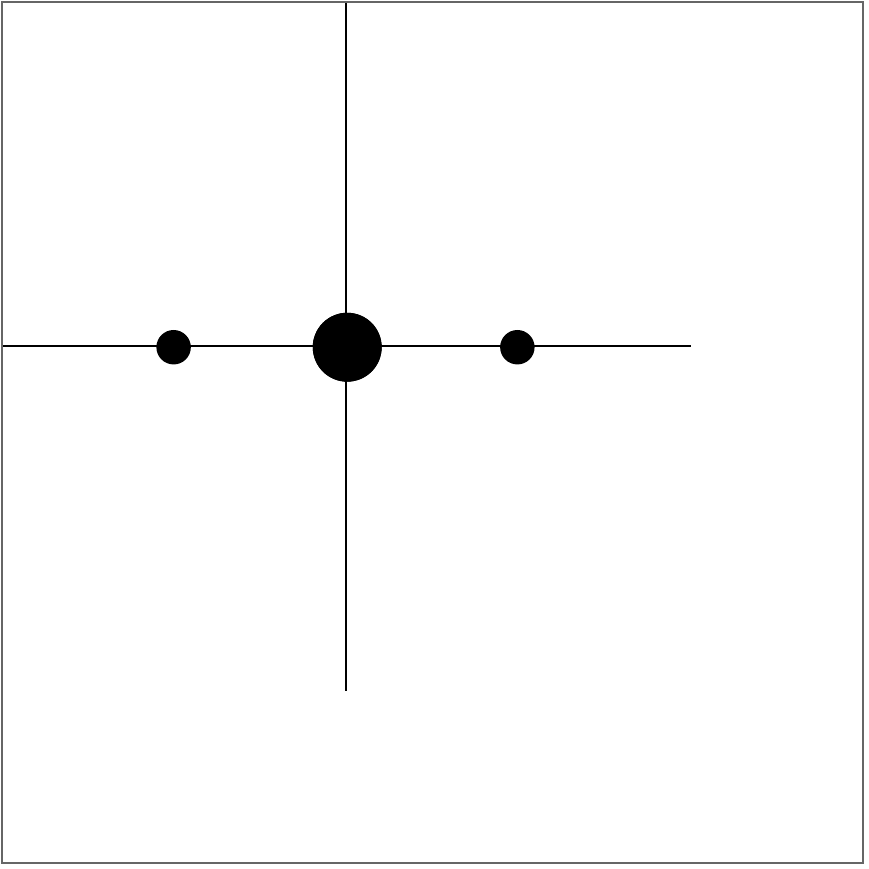}};
 \node at (8,0) {\includegraphics[width=86pt]{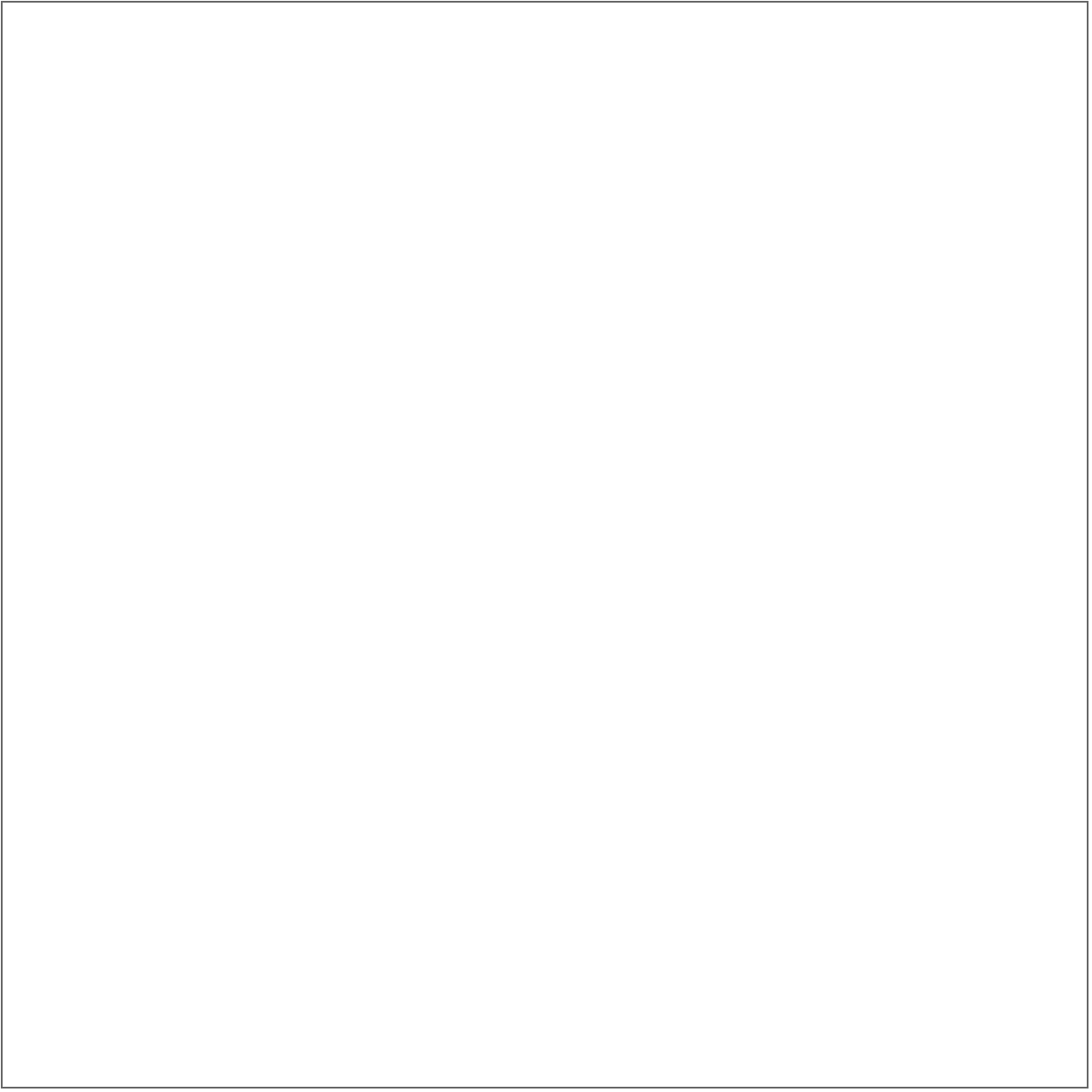}};
 \node at (4,12) {\includegraphics[width=86pt]{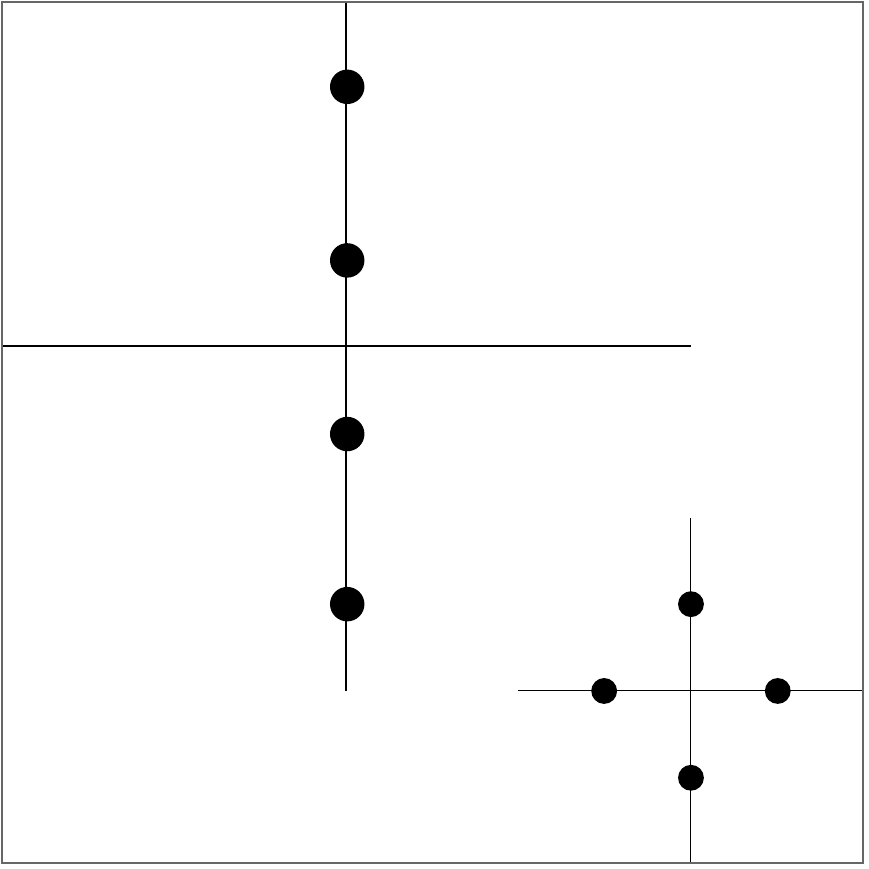}};
 \node at (8,12) {\includegraphics[width=86pt]{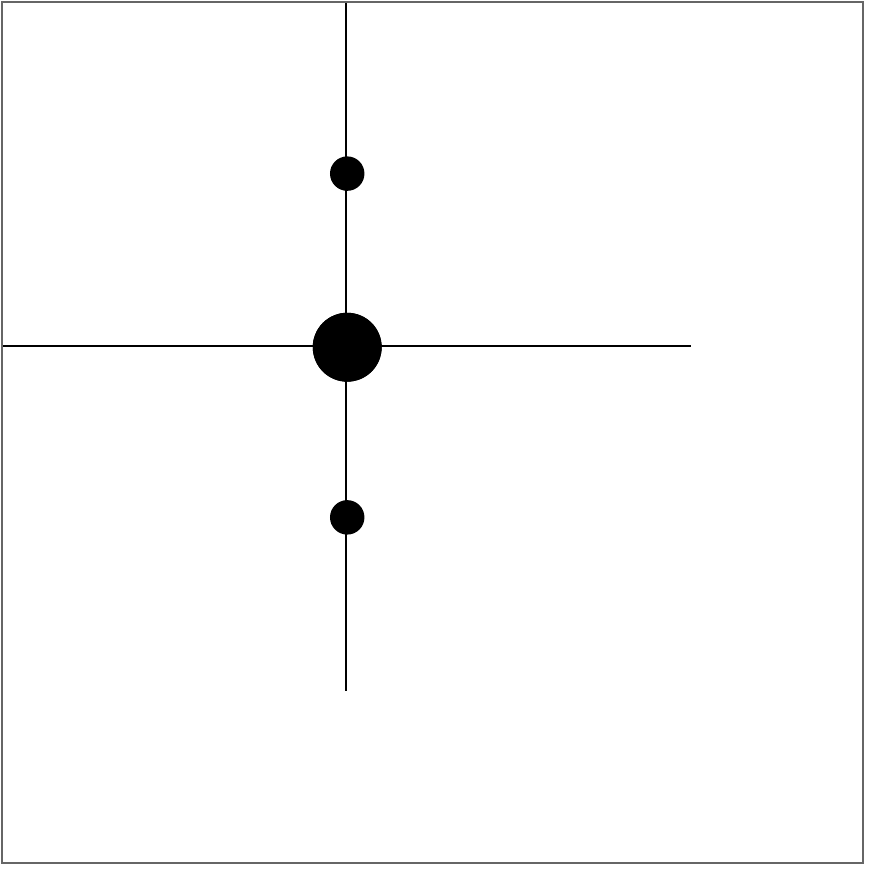}};
 \node at (8,0) {No equilibria};
\end{tikzpicture}
\caption{The versal unfolding of the Hamiltonian \eqref{eq:nilpot_Ham_general}.}
\label{fig:hopf-fold-real}
\end{figure}

Since we are interested only in the local behaviour, we will truncate the Hamiltonian to order 3.
Then the equations of motion become
\begin{equation*}
\begin{split}
\dot{q}_1&=\frac{3 a_3 q_1^2}{2}-q_2, \\
\dot{q}_2&=p_2+\frac{3 \nu q_1}{4}+\frac{3 a_2 q_1^2}{4}+\frac{3 a_3 q_1 q_2}{2}, \\
\dot{p}_1&=-\kappa-\frac{3 \nu p_2}{4}-\mu q_1-3 a_3 p_1 q_1-\frac{3 a_2 p_2 q_1}{2}-\frac{a_1 q_1^2}{2}-\frac{3 a_3 p_2 q_2}{2}-\frac{a_2 q_2^2}{2}, \\
\dot{p}_2&=p_1-\frac{3 a_3 p_2 q_1}{2}-\nu q_2-a_2 q_1 q_2-2 a_3 q_2^2.
\end{split}
\end{equation*}
If $(q_{1}^*,q_{2}^*,p_{1}^*,p_{2}^*)$ is an equilibrium of the above equations, it holds that
\label{eq:nilpot_ham_sys}
\begin{equation}
\begin{split}
&q_{1}^*=q_0, \\
&q_{2}^*=\frac{3 }{2}a_3 q_0^2, \\
&p_{1}^*=\frac{3}{8}( a_3 \nu q_0^2 + a_2 a_3 q_0^3 + 3 a_3^3 q_0^4), \\
&p_{2}^*=-\frac{3 }{4}(\nu q_0+a_2 q_0^2+3 a_3^2 q_0^3),
\end{split}
\label{eq:general-equilibrium}
\end{equation}
with $q_0$ satisfying
\begin{equation}
\frac{27}{16} a_3^4 q_0^5+\frac{45}{16} a_2 a_3^2 q_0^4+\frac{9}{4} a_3^2 \nu q_0^3+\frac{9}{8} a_2^2 q_0^3+\frac{27}{16} a_2 \nu q_0^2- \frac{1}{2}a_1 q_0^2 +\frac{9}{16} \nu^2 q_0-\mu q_0-\kappa=0.
\label{eq:equilibrium-surface}
\end{equation}

\subsection{Reparametrization}

In order to obtain the equilibria of the Hamiltonian system, one has to solve the 5th order polynomial \eqref{eq:equilibrium-surface}.
However, there exists no algebraic formula for its roots.
Thus, we proceed by changing the parameters of the system.

First we observe that equation \eqref{eq:equilibrium-surface} defines a smooth hyper-surface, which will be called \textit{surface of equilibria}, in the 4-dimensional space spanned by $(\kappa,\mu,\nu,q_0)$.
Since there exists a global coordinate chart $(\mu,\nu,q_0)$, we view the above equation as the definition of a function
\begin{equation}
\begin{split}
 \kappa(\mu,\nu,q_0)&=\frac{27}{16} a_3^4 q_0^5+\frac{45}{16} a_2 a_3^2 q_0^4+\frac{9}{4} a_3^2 \nu q_0^3+\frac{9}{8} a_2^2 q_0^3 \\
 & \hphantom{=} +\frac{27}{16} a_2 \nu q_0^2- \frac{1}{2}a_1 q_0^2 +\frac{9}{16} \nu^2 q_0-\mu q_0.
\end{split}
\label{eq:kappa-equals}
\end{equation}
By replacing $\kappa$ with the right hand side of equation \eqref{eq:kappa-equals}, equations \eqref{eq:general-equilibrium} still define an equilibrium.
However now $q_0$ is not viewed as a value that needs to be computed, but rather as a parameter on its own.
So we shift the parametrization of the system from $(\kappa,\mu,\nu)$ to $(\mu,\nu,q_0)$ and we restrict our analysis in a neighbourhood of the origin in the space spanned by $(\mu,\nu,q_0)$.

One possible problem with this method is that $q_0$ is not guaranteed to be small when $\kappa$, $\mu$ and $\nu$ are small.
Setting $\kappa=\mu=\nu=0$ gives
\begin{equation*}
q_0^2\Bigg(\frac{27}{16} a_3^4 q_0^3 + \frac{45}{16} a_2 a_3^2 q_0^2 + \frac{9}{8} a_2^2 q_0 - \frac{1}{2}a_1 \Bigg)=0 .
\end{equation*}
We see that $q_0=0$ is a double root.
This implies that in any neighbourhood of the origin in $(\mu,\nu,\kappa)$ there are always at least 2 possible values for $q_0$.
So we are indeed able to restrict our analysis in a region where $q_0$ is small.

Notice that each triplet $(\mu,\nu,q_0)$ defines exactly one equilibrium and each equilibrium can be described by one such triplet.
Also each such triplet defines uniquely the Hamiltonian $H_{\kappa,\mu,\nu}$.
So there is a bijection between the set of the Hamiltonian systems in this family paired with one of its equilibrium points and the points on the surface, $\{H_{\kappa,\mu,\nu},q_0\}\leftrightarrow(\kappa,\mu,\nu,q_0)$. 
The eigenvalues of this equilibrium are 
\begin{equation*}
\pm \frac{\sqrt{2}}{4} \sqrt{Q(\mu,\nu,q_0)\pm\sqrt{P(\mu,\nu,q_0)}},
\end{equation*}
with
\begin{align*}
Q(\mu,\nu,q_0) = & 3 a_3^2 q_0^2-16 a_2 q_0-10 \nu \\
P(\mu,\nu,q_0) = & -531 a_3^4 q_0^4-816 a_2 a_3^2 q_0^3-492 a_3^2 \nu q_0^2+40 a_2^2 q_0^2 \\
& +104 a_2 \nu q_0+64 \nu^2+64 a_1 q_0+64 \mu. \\
\end{align*}
Using the eigenvalues, one may search for Hamiltonian-Hopf and centre-saddle bifurcations.

\subsection{Hamiltonian-Hopf bifurcation}

If $Q(\mu,\nu,q_0)$ is negative and $P(\mu,\nu,q_0)$ changes sign, then the eigenvalues change in the same way as the eigenvalues of a system undergoing a Hamiltonian-Hopf bifurcation. So the next step is to seek whether the bifurcation actually takes place. 

Let $Q(\mu,\nu,q_0)=3 a_3^2 q_0^2-16 a_2 q_0-10 \nu=-\omega^2$, with $\omega>0$. By taking $P(\mu,\nu,q_0)=0$ and solving it with respect to $\mu$, it yields
\begin{equation}
\begin{split}
\mu_h =-a_1 q_0-\nu^2-\frac{13}{8} a_2 \nu q_0 - \frac{5}{8} a_2^2 q_0^2 +\frac{123}{16} a_3^2 \nu q_0^2+\frac{51}{4} a_2 a_3^2 q_0^3+\frac{531}{64} a_3^4 q_0^4.
\end{split}
\label{ch2:eq:hopf-subst}
\end{equation}
By substituting $\mu$ by the above, one gets equations of motion with a non-semi-simple linear part and eigenvalues $\pm\ii\omega(2\sqrt{2})^{-1}$.

Let $J'$ be the Jacobian matrix at the equilibrium and let
\begin{equation*}
J= \left( \begin{array}{cccc}
         0 & \frac{\omega}{2\sqrt{2}} & 0 & 0 \\
         -\frac{\omega}{2\sqrt{2}} & 0 & 0 & 0 \\
         -1 & 0 & 0 & \frac{\omega}{2\sqrt{2}} \\
         0 & -1 & -\frac{\omega}{2\sqrt{2}} & 0
         \end{array}\right).
\end{equation*}
Then the system $P\, J=J'\, P$ and $P^\intercal\, \Omega\, P=\Omega$ can be solved\footnote{Recall that $\Omega$ is the $2m \times 2m$ matrix $\left( \begin{smallmatrix} 0&E\\ -E&0 \end{smallmatrix} \right)$.}. One solution is the matrix
\begin{equation*}
P^\intercal=\left( \begin{smallmatrix}
         0 & \frac{1}{ \omega} & -\frac{3 \left(4 a_2 q_0-87 a_3^2 q_0^2+6 \omega^2\right)}{40\omega } & -\frac{9 a_3 q_0}{2 \omega} \\
         \frac{ 2\sqrt{2}}{ \omega^2} & \frac{6 \sqrt{2} a_3 q_0 }{ \omega^2} & -\frac{3 a_3 q_0 \left(a_2 q_0-63 a_3^2 q_0^2-16 \omega^2\right)}{10 \sqrt{2}\omega^2} & -\frac{3 \left(4 a_2 q_0+93 a_3^2 q_0^2+6 \omega^2\right)}{10 \sqrt{2}\omega^2} \\
         \frac{2}{ \omega} & \frac{6 a_3 q_0 }{ \omega} & -\frac{3a_3 q_0\left(a_2 q_0-63 a_3^2 q_0^2+4 \omega^2\right)}{20\omega }  & -\frac{12 a_2 q_0+279 a_3^2 q_0^2-2 \omega^2}{20\omega} \\
         0 & \frac{1}{\sqrt{2}} & \frac{-12 a_2 q_0+261 a_3^2 q_0^2+2 \omega^2}{40 \sqrt{2}} & -\frac{9 a_3 q_0}{2 \sqrt{2}}
         \end{smallmatrix}\right).
\end{equation*}
Shifting the axes so that the equilibrium is always at zero and using the above transformation, the linear part of the Hamiltonian can be transformed to $$\frac{ \omega}{2 \sqrt{2}}(p_2 q_1-p_1 q_2)+\frac{1}{2} \left(q_1^2+q_2^2\right).$$
Now instead of using equation \eqref{ch2:eq:hopf-subst}, we use
\begin{equation*}
\begin{split}
\mu=\beta+\mu_h,
\end{split}
\end{equation*}
where the unfolding parameter $\beta$ is added.
With the same transformation as above the linear part of Hamiltonian becomes $$\frac{ \omega}{2 \sqrt{2}}(p_2 q_1-p_1 q_2)+\frac{1}{2} \left(q_1^2+q_2^2\right)+\frac{2 \beta }{\omega^4}(2 q_2^2+2 \sqrt{2} \omega p_1 q_2+ \omega^2p_1^2 ). $$
Then the Hamiltonian is normalized with the algorithm described in \cite{vdMeer82}. The parameter $\beta$ is counted for the degree of the monomial. After the normalization, the linear part of the Hamiltonian becomes 
\begin{equation*}
\begin{split}
H_2=&\left(\frac{ \omega}{2 \sqrt{2}}-\frac{2 \sqrt{2} \beta}{\omega^3}-\frac{48 \sqrt{2} \beta^2}{\omega^7}\right) \left(p_2 q_1-p_1 q_2\right) \\
&+ \left(\frac{1}{2} +\frac{2 \beta}{\omega^4}+\frac{96 \beta^2}{\omega^8}\right) \left(q_1^2+q_2^2\right) + \left(\frac{ \beta}{\omega^2}+\frac{16 \beta^2}{\omega^6}\right) \left(p_1^2+p_2^2\right).
\end{split}
\end{equation*}
Let $C_m$ be the coefficient of $M^2$ in the normal form. It satisfies
\begin{equation}
\begin{split}
\omega^8C_m=& \frac{320}{3} a_1^2-288 a_1 a_2^2 q_0+\frac{972}{5} a_2^4 q_0^2-2808 a_1 a_2 a_3^2 q_0^2+\frac{18954}{5} a_2^3 a_3^2 q_0^3 \\
&-8064 a_1 a_3^4 q_0^3 +\frac{587331}{20} a_2^2 a_3^4 q_0^4+\frac{530712}{5} a_2 a_3^6 q_0^5+\frac{762048}{5} a_3^8 q_0^6 \\
&+\frac{16}{3} a_1 a_2 \omega^2-\frac{36}{5} a_2^3 q_0 \omega^2-344 a_1 a_3^2 q_0 \omega^2+\frac{11907}{50} a_2^2 a_3^2 q_0^2 \omega^2 \\
&+\frac{180513}{50} a_2 a_3^4 q_0^3 \omega^2+\frac{609093}{50} a_3^6 q_0^4 \omega^2-\frac{36}{5} a_2^2 \omega^4-\frac{682}{25} a_2 a_3^2 q_0 \omega^4 \\
&+\frac{25959}{100} a_3^4 q_0^2 \omega^4+\frac{1328}{75} a_3^2 \omega^6.
\end{split}
\label{eq:m-coefficient}
\end{equation}
In order to determine the type of the Hamiltonian-Hopf bifurcation the sign of $C_m$ is needed.
The polynomial \eqref{eq:m-coefficient} has the same sign as $C_m$.
However, we are interested only in what is happening around 0 in parameters.
Since $\omega^2$ is of the same order as $q_0$ and $\nu$, only the term $\frac{80}{3} a_1^2$ determines the sign of $C_m$.
This implies that a Hamiltonian-Hopf bifurcation actually happens in the system and it is of the supercritical type as long as $a_1$ does not vanish.

\subsection{Centre-saddle bifurcation}

There are two ways we can search for the centre-saddle bifurcation.
One is to look at the surfaces $Q(\mu,\nu,q_0)^2=P(\mu,\nu,q_0)$ and $Q(\mu,\nu,q_0)^2=-P(\mu,\nu,q_0)$.
Another is to search where both $\partial_{q_0}\kappa_{\mu,\nu,q_0}=0$ and $\partial^2_{q_0}\kappa_{\mu,\nu,q_0}\ne0$ hold, with $\kappa_{\mu,\nu,q_0}$ given by equation \eqref{eq:kappa-equals}.
The latter gives the relation
\begin{equation}
\mu = \frac{9}{16} \nu^2-a_1 q_0+\frac{27}{8} a_2 \nu q_0+\frac{27}{8} a_2^2 q_0^2+\frac{27}{4} a_3^2 \nu q_0^2+\frac{45}{4} a_2 a_3^2 q_0^3+\frac{135}{16} a_3^4 q_0^4,\;\;\;\;
\label{eq:fold-surface}
\end{equation}
while
\begin{equation*}
\partial^2_{q_0}\kappa_{\mu,\nu,q_0}=-a_1+\frac{27}{8} a_2 \nu+\frac{27}{4} a_2^2 q_0+\frac{27}{2} a_3^2 \nu q_0+\frac{135}{4} a_2 a_3^2 q_0^2+\frac{135}{4} a_3^4 q_0^3.
\end{equation*}
This shows that there exists a neighbourhood around zero in parameters, in which the above second derivative does not vanish as long as $a_1$ does not vanish.

Substituting $\mu$ by the relation \eqref{eq:fold-surface}, one finds that the equilibrium has eigenvalues $0$, corresponding to a Jordan block of degree 2, and $\pm \sqrt{\lambda}$ with $\lambda= -20 \nu-32 a_2 q_0+6 a_3^2 q_0^2$.
So there are two distinct possibilities for the non-zero eigenvalues.
We will see that in both cases a centre-saddle bifurcation actually happens.

\subsubsection{Hyperbolic eigenvalues}

If $\lambda>0$, the non-zero eigenvalues of the equilibrium are $\pm \omega$, where $\omega\in\mathbb{R}$ is such that $\omega^2=\lambda$.
Then the system can be transformed to one having an equilibrium with its Jacobian matrix being
\begin{equation*}
J_1=\left(
\begin{array}{cccc}
 0 & 0 & 1 & 0 \\
 0 & 0 & 0 & \omega \\
 0 & 0 & 0 & 0 \\
 0 & \omega & 0 & 0
\end{array}
\right).
\end{equation*}
Let $J_1'$ be the Jacobian matrix at the equilibrium. Then one may search for a matrix $P$ that satisfies $P_1\, J_1=J_1'\, P_1$ and $P_1^\intercal\, \Omega\, P_1=\Omega$. One such matrix is
\begin{equation*}
P_1^\intercal=\left(
\begin{smallmatrix}
 -\frac{1}{\omega} & -\frac{3 a_3 q_0}{\omega} & \frac{3 a_3 q_0 \left(a_2 q_0-63 a_3^2 q_0^2+4 \omega^2\right)}{40 \omega} & \frac{3 \left(4 a_2 q_0+93 a_3^2 q_0^2-4 \omega^2\right)}{40 \omega} \\
 0 & -\frac{1}{\sqrt{\omega}} & -\frac{3 \left(-4 a_2 q_0+87 a_3^2 q_0^2+4 \omega^2\right)}{40 \sqrt{\omega}} & \frac{9 a_3 q_0}{2 \sqrt{\omega}} \\
 0 & \frac{1}{\omega} & -\frac{12 a_2 q_0-261 a_3^2 q_0^2+28 \omega^2}{40 \omega} & -\frac{9 a_3 q_0}{2 \omega} \\
 \frac{1}{\omega^{3/2}} & \frac{3 a_3 q_0}{\omega^{3/2}} & \frac{3 a_3 q_0 \left(-a_2 q_0+63 a_3^2 q_0^2+36 \omega^2\right)}{40 \omega^{3/2}} & \frac{-12 a_2 q_0-279 a_3^2 q_0^2-28 \omega^2}{40 \omega^{3/2}}
\end{smallmatrix}
\right).
\end{equation*}
With a shift in the axes so that the equilibrium is always at zero and using the above matrix as transformation, the linear part of the Hamiltonian can be transformed to $\frac{1}{2}p_1^2+\frac{1}{2} \omega\left(p_2^2-q_2^2\right)$.
Then for $C_{q_1^3}^+$, the coefficient of $q_1^3$ in the Hamiltonian, it holds
\begin{equation*}
\omega^3 C_{q_1^3}^+=-\frac{a_1}{6}+\frac{9}{40} a_2^2 q_0+\frac{351}{160} a_2 a_3^2 q_0^2+\frac{63}{10} a_3^4 q_0^3-\frac{9}{40} a_2 \omega^2-\frac{9}{10} a_3^2 q_0 \omega^2.
\end{equation*}
Since $\omega^2$ is of the same order as $q_0$ and $\nu$, $C_{q_1^3}^+$ does not vanish as long as $a_1$ does not vanish.

\subsubsection{Elliptic eigenvalues}

If $\lambda<0$, the non-zero eigenvalues of the equilibrium are $\pm \ii\omega$, where $\omega\in\mathbb{R}$ is such that $\omega^2=-\lambda$.
Then the system can be transformed to one having equilibrium with Jacobian matrix being
\begin{equation*}
J_2=\left(
\begin{array}{cccc}
 0 & 0 & -1 & 0 \\
 0 & 0 & 0 & \omega \\
 0 & 0 & 0 & 0 \\
 0 & -\omega & 0 & 0
\end{array}
\right).
\end{equation*}
Let $J_2'$ be the Jacobian matrix at the equilibrium, then one may search for a matrix $P_2$ that satisfies $P_2\, J_2=J_2'\, P_2$ and $P_2^\intercal\, \Omega\, P_2=\Omega$. One such matrix is
\begin{equation*}
P_2^\intercal=\left(
\begin{smallmatrix}
 \frac{1}{\omega} & \frac{3 a_3 q_0}{\omega} & -\frac{3 a_2 a_3 q_0^2}{40 \omega}+\frac{189 a_3^3 q_0^3}{40 \omega}+\frac{3 a_3 q_0 \omega}{10} & -\frac{3 a_2 q_0}{10 \omega}-\frac{279 a_3^2 q_0^2}{40 \omega}-\frac{3 \omega}{10} \\
 0 & \frac{1}{\sqrt{\omega}} & -\frac{3 a_2 q_0}{10 \sqrt{\omega}}+\frac{261 a_3^2 q_0^2}{40 \sqrt{\omega}}-\frac{3 \omega^{3/2}}{10} & -\frac{9 a_3 q_0}{2 \sqrt{\omega}} \\
 0 & \frac{1}{\omega} & -\frac{3 a_2 q_0}{10 \omega}+\frac{261 a_3^2 q_0^2}{40 \omega}+\frac{7 \omega}{10} & -\frac{9 a_3 q_0}{2 \omega} \\
 \frac{1}{\omega^{3/2}} & \frac{3 a_3 q_0}{\omega^{3/2}} & -\frac{3 a_2 a_3 q_0^2}{40 \omega^{3/2}}+\frac{189 a_3^3 q_0^3}{40 \omega^{3/2}}-\frac{27}{10} a_3 q_0 \sqrt{\omega} & -\frac{3 a_2 q_0}{10 \omega^{3/2}}-\frac{279 a_3^2 q_0^2}{40 \omega^{3/2}}+\frac{7 \sqrt{\omega}}{10}
\end{smallmatrix}
\right).
\end{equation*}
With a shift in the axes so the equilibrium is always at zero and using the above matrix as transformation, the linear part of the Hamiltonian can be transformed to -$\frac{1}{2} p_1^2+\frac{1}{2}\omega\left(p_2^2+q_2^2\right) $.
Then if $C_{q_1^3}^-$ is the coefficient of $q_1^3$ in the Hamiltonian, it holds
\begin{equation*}
\omega^3 C_{q_1^3}^-=-\frac{a_1}{6}+\frac{9}{40} a_2^2 q_0+\frac{351}{160} a_2 a_3^2 q_0^2+\frac{63}{10} a_3^4 q_0^3-\frac{9}{40} a_2 \omega^2-\frac{9}{10} a_3^2 q_0 \omega^2.
\end{equation*}
Exactly as above, $C_{q_1^3}^-$ does not vanish as long as $a_1$ does not vanish.

\subsection{Nilpotent equilibrium}

We have found so far that the equilibria of the system are on the 3-dimensional surface of equilibria living in the 4-dimensional parameter space $(\kappa,\mu,\nu,q_0)$, defined by equation \eqref{eq:equilibrium-surface}.
Moreover there is a 2-dimensional surface living on the surface of equilibria, defined by equation \eqref{eq:fold-surface}, on which the system undergoes a centre-saddle or fold bifurcation, when $a_1\ne0$.
This surface will be called \textit{fold surface}
Because of the fold bifurcation, the Hamiltonian has different values on the two equilibria close to the origin.
This implies that there can be no heteroclinic connections between them.

There is also another 2-dimensional surface living on the surface of equilibria, defined by $P(\kappa,\mu,\nu,q_0)=0$, on which the system undergoes a Hamiltonian-Hopf bifurcation if the eigenvalues are not real and $a_1\ne0$, this will be called \textit{Hopf surface}.
It should be stressed here that the Hamiltonian-Hopf bifurcation does not happen on the whole Hopf surface.
As we will see in the next sections there are 2 possible transitions for a system passing through the Hopf surface.
One is the actual Hamilton-Hopf bifurcation, where all the eigenvalues are initially imaginary and they become complex.
In the other case, all the eigenvalues are real and they become complex.
In the second case the equilibrium stays unstable throughout the bifurcation, so this transition is not of particular interest to us.

It is clear by the eigenvalue configurations on the two aforementioned surfaces, that the eigenvalues of the equilibrium will vanish when the two surfaces meet.
The two surfaces are tangent along the line defined by $\nu= -\frac{8}{5} a_2 q_0+\frac{3}{10} a_3^2 q_0^2$ and they do not meet anywhere else in a neighbourhood of the origin.
On that line it can be checked that the equilibrium has a nilpotent Jacobian.

So we see that the nilpotent equilibrium happens on a line.
This is of course due to the fact that there are three parameters in our system, instead of the two that the linear unfolding requires.

\subsection{The effect of the coefficients}

Let $\kappa_{a_1,a_2,a_3}$, $f_{a_1,a_2,a_3}$ and $h_{a_1,a_2,a_3}$ denote the surface of equilibria, the fold surface and the Hopf surface, respectively.
Let $f^\kappa_{a_1,a_2,a_3}$, $h^\kappa_{a_1,a_2,a_3}$ denote the projection to $(\kappa,\mu,\nu)$-hyperplane
and $f^{q_0}_{a_1,a_2,a_3}$, $h^{q_0}_{a_1,a_2,a_3}$ denote the projection to the $(\mu,\nu,q_0)$-hyperplane. Lastly, the absence of an index implies that it is zero, for example $\kappa_{a_1,a_2}\equiv\kappa_{a_1,a_2,0}$ and $\kappa_{a_1}\equiv\kappa_{a_1,0,0}$.

\begin{lemma}
If $a_1\ne0$, the surfaces $f^\kappa_{a_1,a_2,a_3}$ and $h^\kappa_{a_1,a_2,a_3}$ are diffeomorphic to the surfaces $f^\kappa_{a_1}$ and $h^\kappa_{a_1}$, respectively.
\end{lemma}

Since the surfaces $\kappa_{a_1,a_2,a_3}$, $f_{a_1,a_2,a_3}$ and $h_{a_1,a_2,a_3}$ are graphs of functions they are trivially diffeomorphic to $\kappa_{a_1}$, $f_{a_1}$ and $h_{a_1}$ respectively.
Then if $a_1\ne0$, there exists a neighbourhood of the origin in which the projection onto $(\kappa,\mu,\nu)$ is smooth.
The above lemma implies that setting $a_2=a_3=0$ does not restrict the genericity of the results.

In a system where $a_1$ vanishes Hamiltonian-Hopf bifurcation of both types can appear.
Also the fold bifurcation and other bifurcations of co-dimension 2 may happen generically.

\section{Truncated Hamiltonian}
\label{ch:truncated-hamiltonian}

Using the above lemma we can set $a_2=a_3=0$ and the Hamiltonian becomes
\begin{equation}
 H = \frac{p_2^2}{2}-p_1 q_2 +\kappa q_1 +\mu \frac{ q_1^2}{2}+\nu\left( \frac{q_2^2}{2}+\frac{3}{4}p_2 q_1 \right)+a_1 \frac{ q_1^3}{6}.
\label{eq:truncated-hamiltonian}
\end{equation}
Then the equilibria satisfy
\begin{equation*}
q_{10}=q_0,\;\; q_{20}=0,\;\; p_{10}=0,\;\; p_{20}=-\frac{3 }{4}\nu q_0.
\end{equation*}
So the equation defining the surface of equilibria takes the much simpler form
\begin{equation*}
- \frac{1}{2}a_1 q_0^2 +\frac{9}{16} \nu^2 q_0-\mu q_0-\kappa=0.
\end{equation*}
The equations can be simplified further, since if $a_1$ is positive a time rescaling can transform $a_1$ to 1. If $a_1$ is negative it can be transformed to -1 and a change of sign on $q_1$ transforms it to 1.

Notice that there can only be two equilibria instead of five.
We already know that qualitative changes of the eigenvalues happen only on the two surfaces, the fold and the Hopf.
Thus for a description of the eigenvalue configurations the two surfaces need to be drawn in the parameter space.

In the forthcoming sections we assume that $a_1=1$.

\paragraph{The Hopf surface} \hspace{0em}

\begin{figure}[p]
  \subfloat[The Hopf surface projected on the $(\mu,\nu,q_0)$-hyperplane.\label{fig:hopf-surf-nat}]{%
  \begin{overpic}[width=0.45\textwidth]{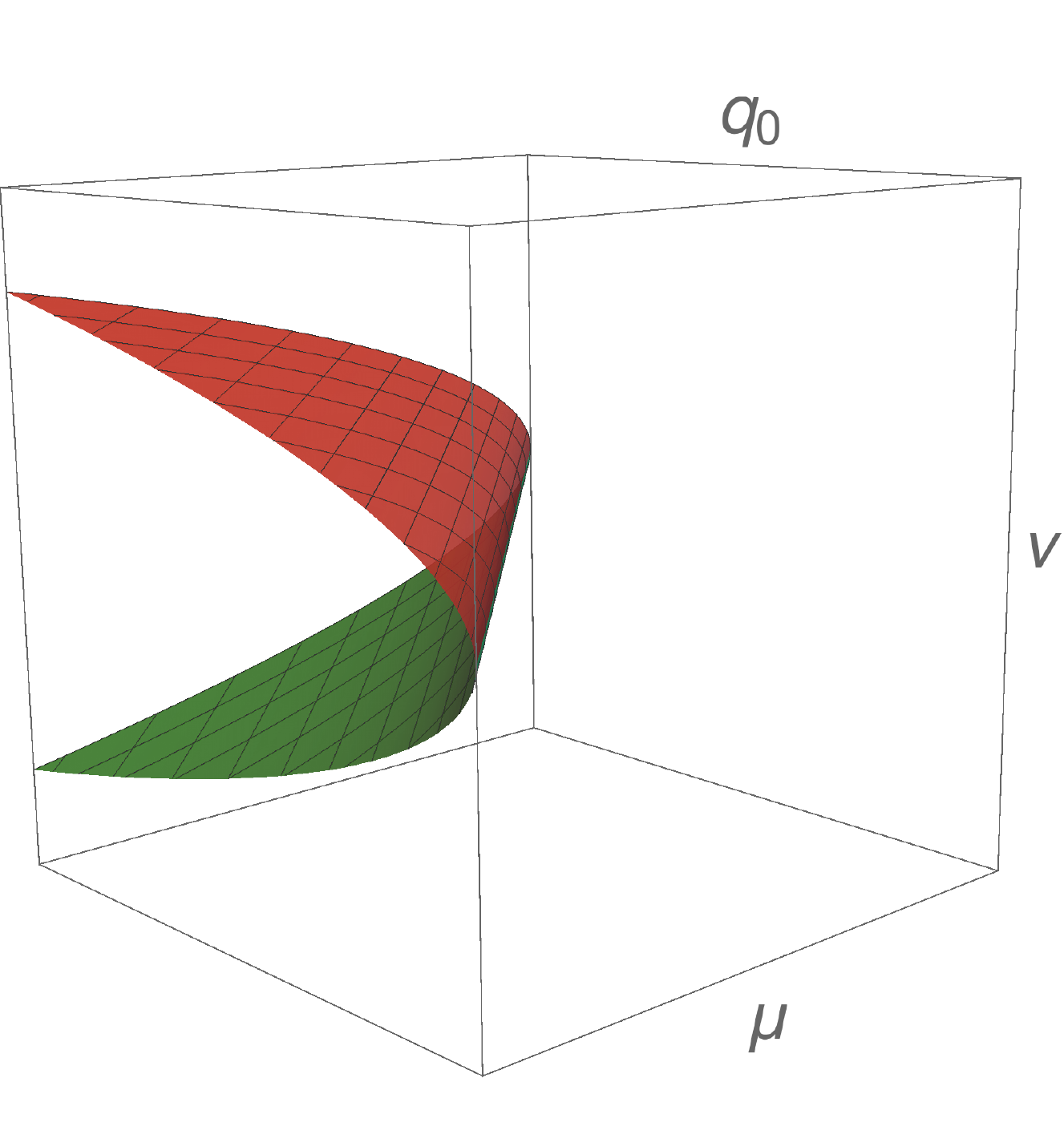}
  \end{overpic}
  }\hfill
  \subfloat[The Hopf surface projected on the $(\kappa,\mu,\nu)$-hyperplane.\label{fig:hopf-surf-real}]{%
  \begin{overpic}[width=0.45\textwidth]{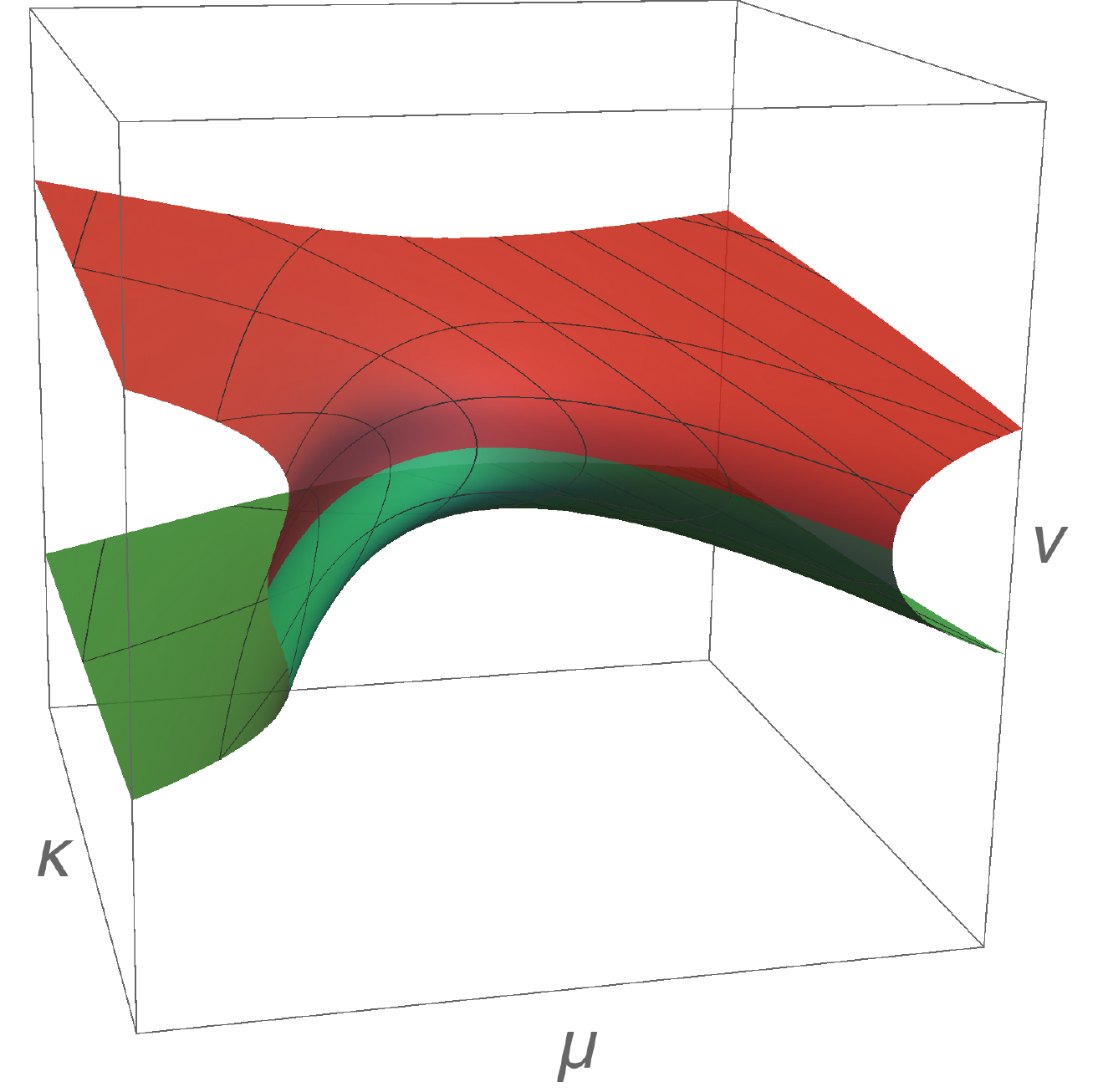}
  \end{overpic}
  }
  \caption{Projections of the Hopf surface.}\label{fig:hopf-surf}
\end{figure}

\begin{figure}[p]
  \subfloat[The fold surface projected on the $(\mu,\nu,q_0)$-hyperplane.\label{fig:fold-surf-nat}]{%
  \begin{overpic}[width=0.45\textwidth]{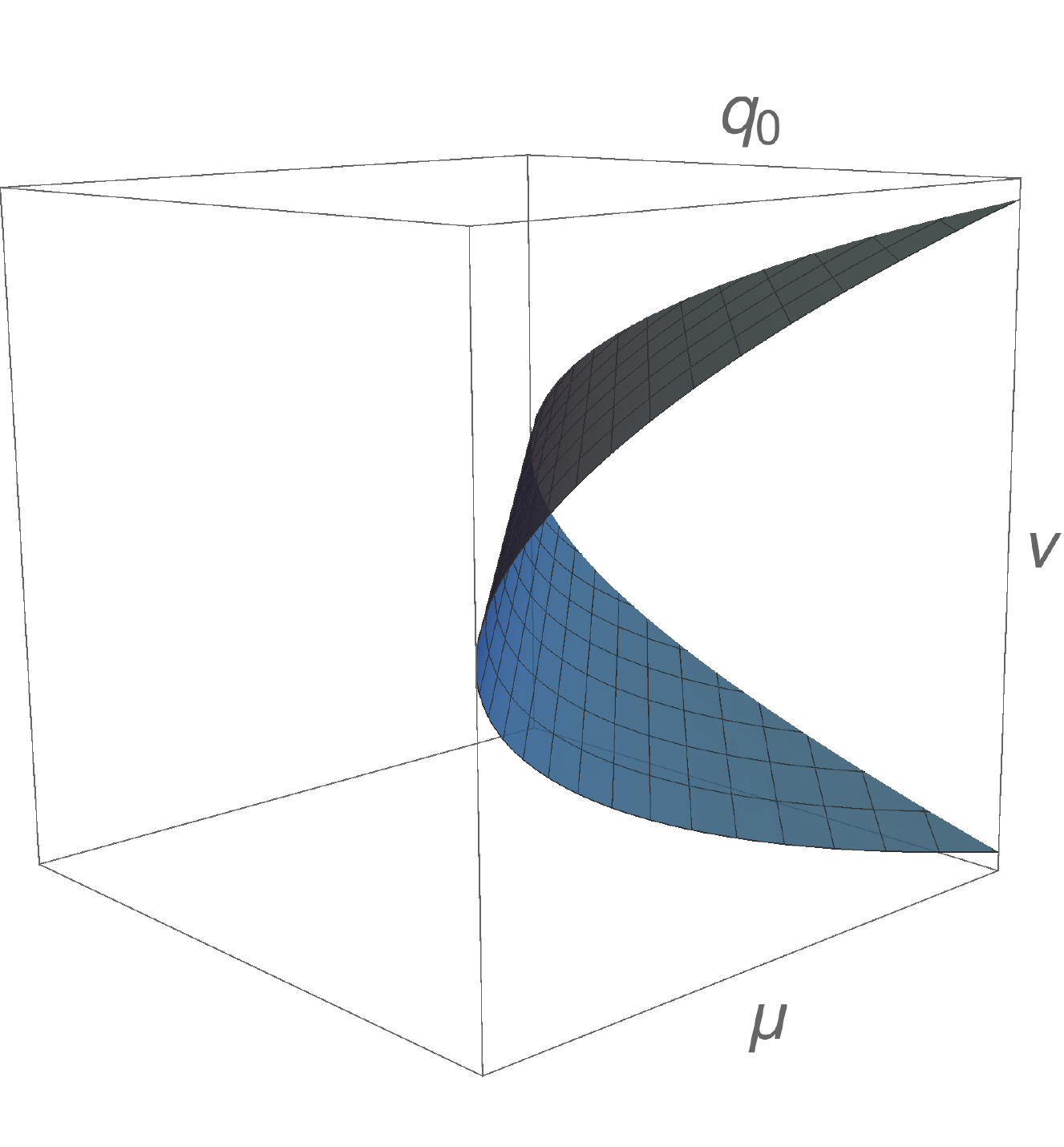}
  \end{overpic}
  }\hfill
  \subfloat[The fold surface projected on the $(\kappa,\mu,\nu)$-hyperplane.\label{fig:fold-surf-real}]{%
  \begin{overpic}[width=0.45\textwidth]{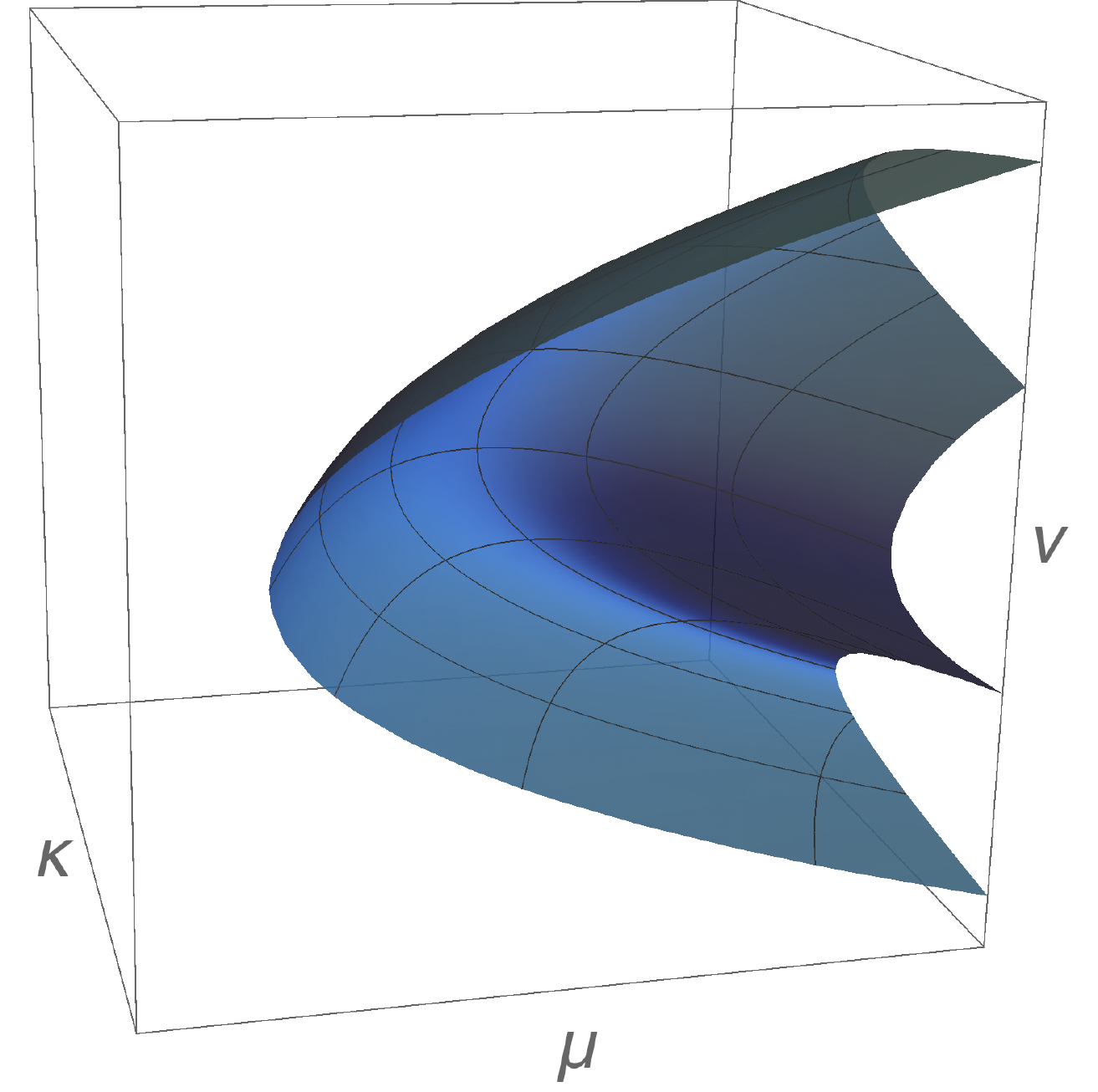}
  \end{overpic}
  }
  \caption{Projections of the fold surface.}\label{fig:fold-surf}
\end{figure}

\begin{figure}[p]
\begin{tikzpicture}
 \node at (6,6) {\includegraphics[width=220pt]{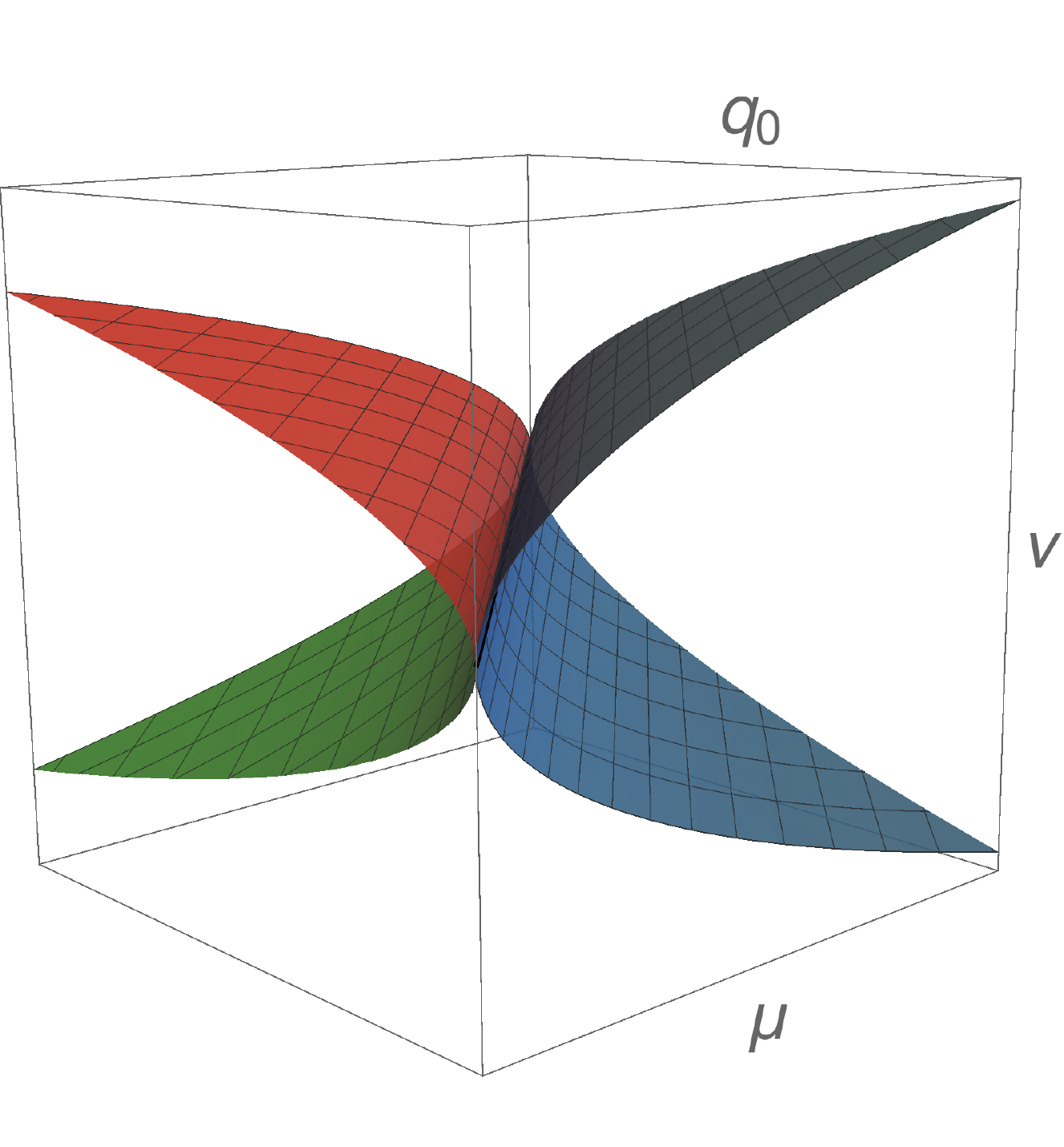}};
 \draw[->] (8.5,10.5) -- (9,8.6);
 \draw[->] (5,10.5) -- (5.8,7.8);
 \draw[->] (1.4,10.5) -- (2.3,8.04);
 \draw[->] (1.5,7) -- (3.5,6.3);
 \draw[->] (1.5,4.5) -- (2.4,4.6);
 \draw[->] (1.5,1.4) -- (5.5,4);
 \draw[->] (5.5,1.5) -- (7.2,4.13);
 \draw[->] (8.5,1.5) -- (8.4,6);
 \node at (0,12) {\includegraphics[width=85pt]{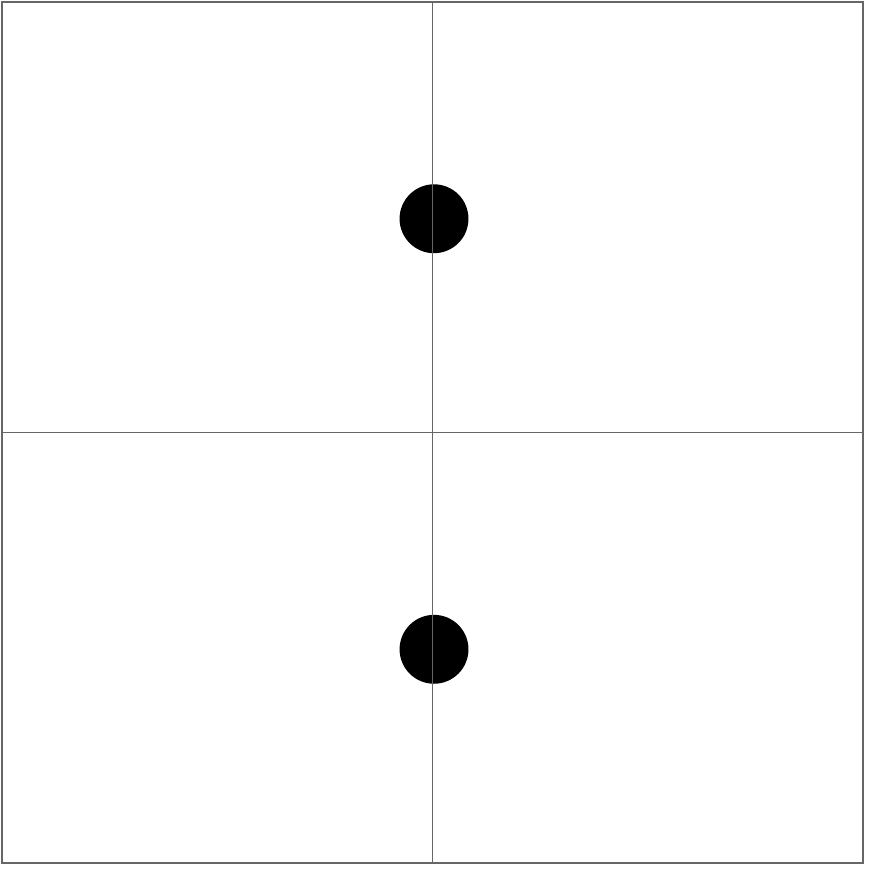}};
 \node at (0,8) {\includegraphics[width=86pt]{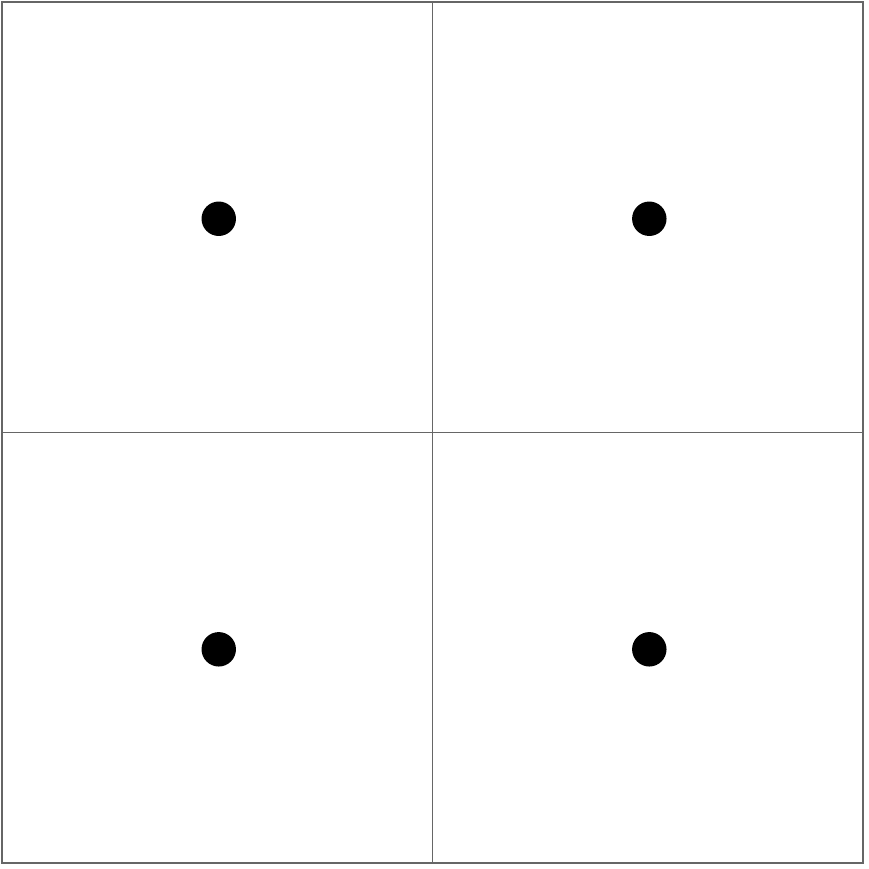}};
 \node at (0,4) {\includegraphics[width=86pt]{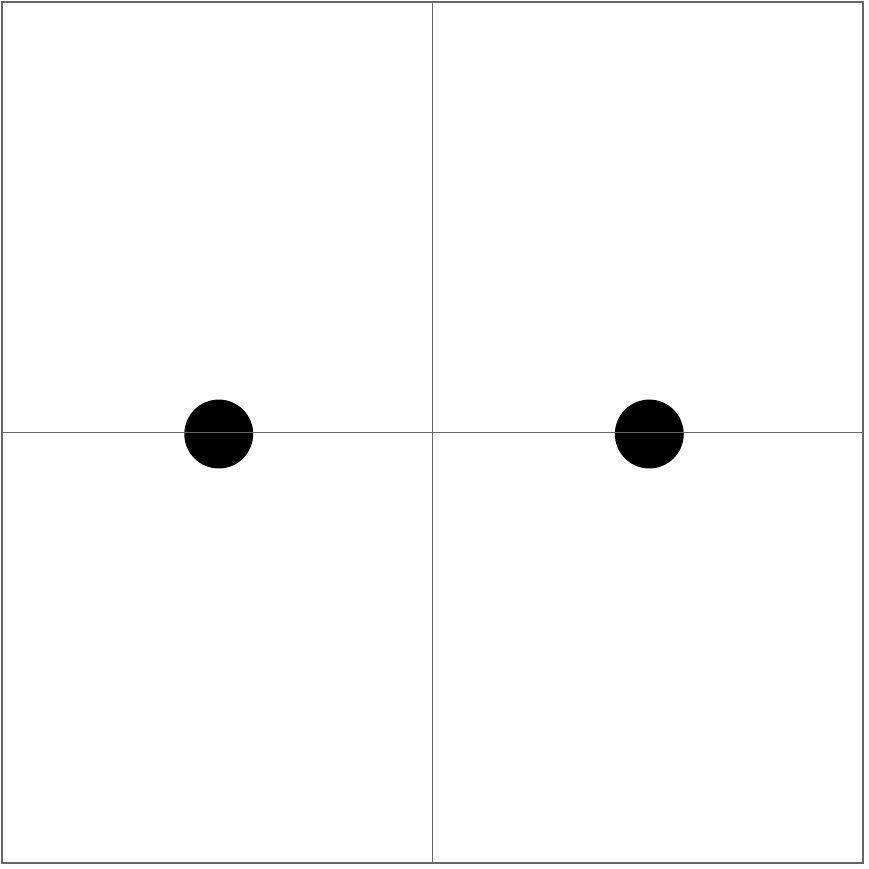}};
 \node at (0,0) {\includegraphics[width=86pt]{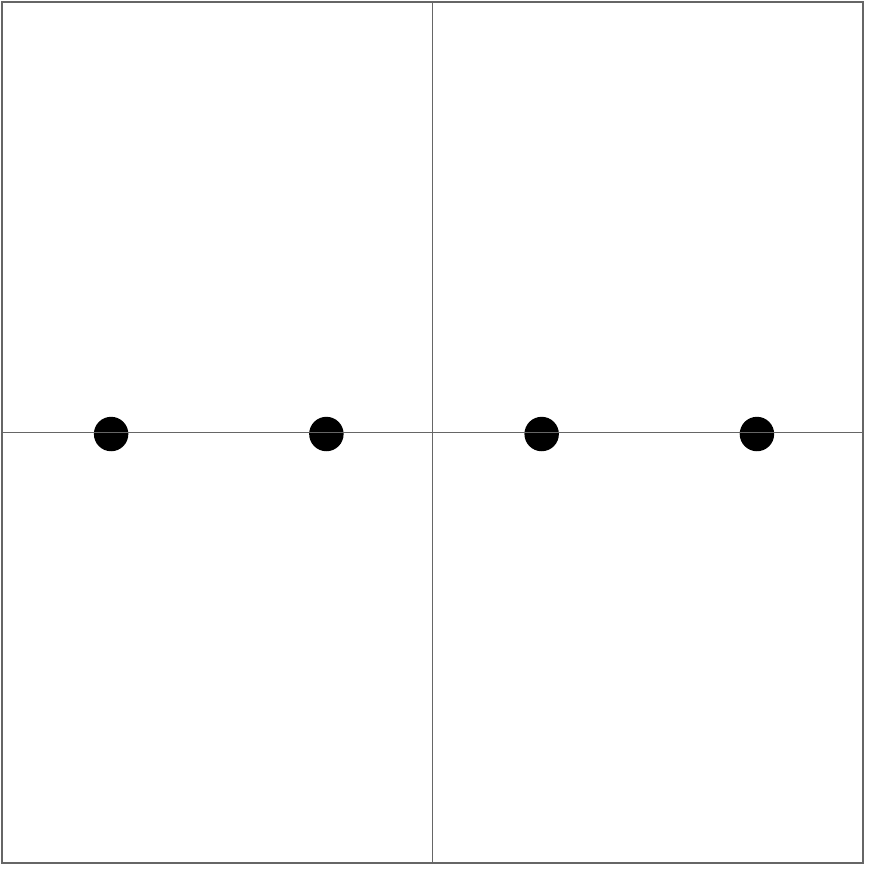}};
 \node at (4,0) {\includegraphics[width=86pt]{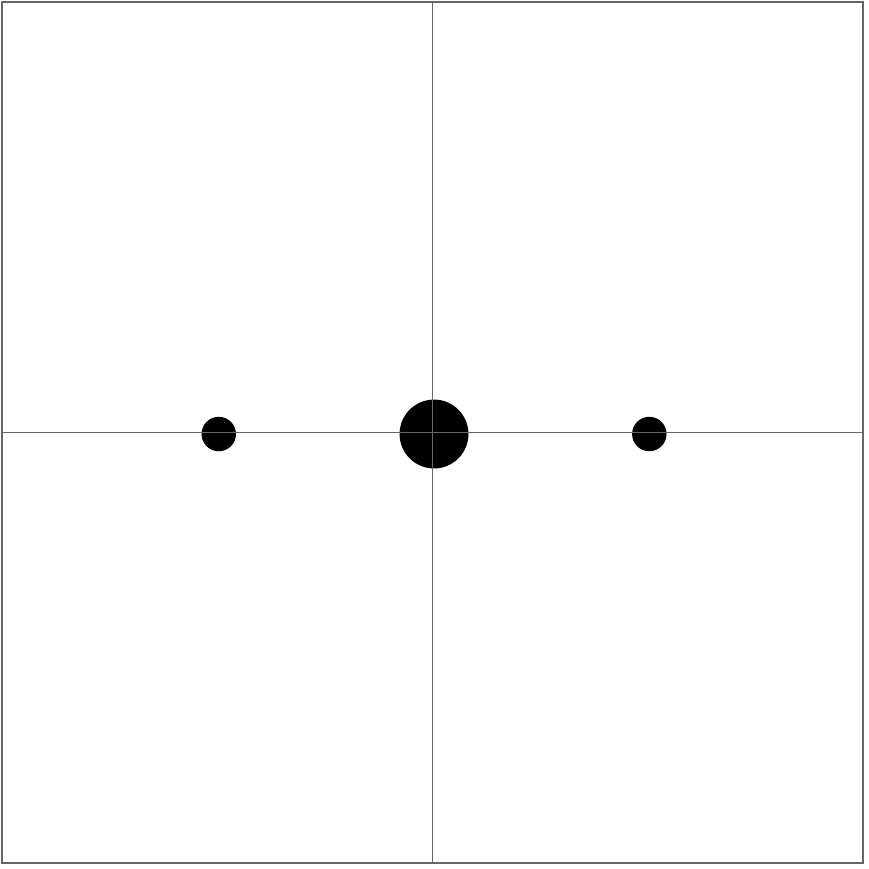}};
 \node at (8,0) {\includegraphics[width=86pt]{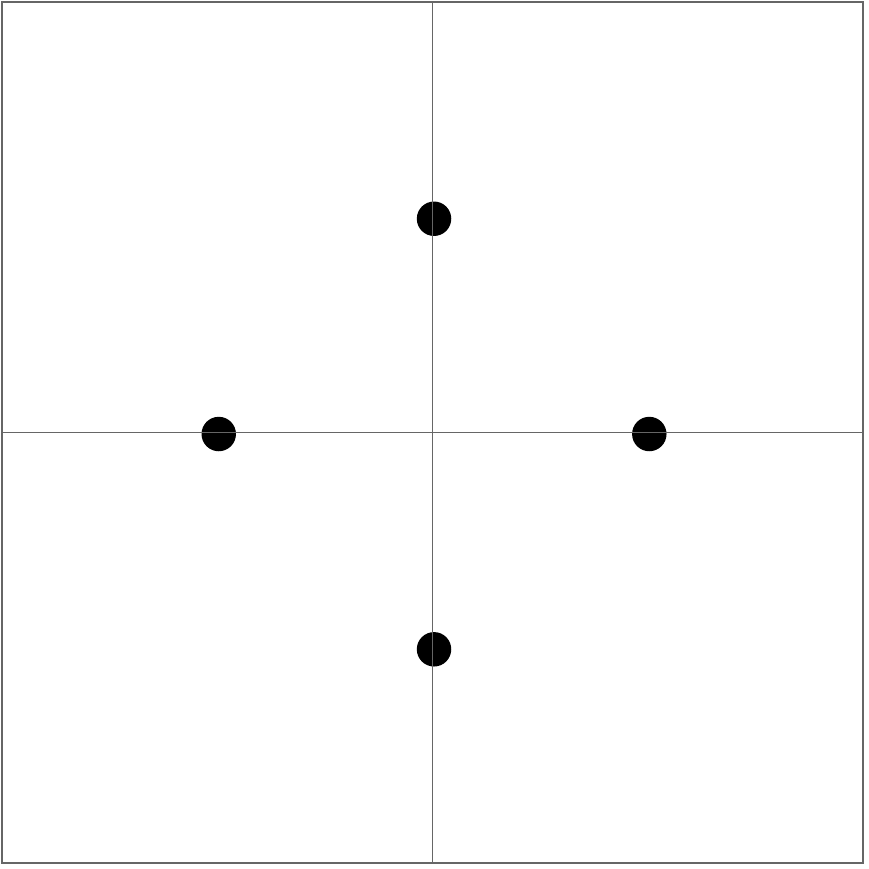}};
 \node at (4,12) {\includegraphics[width=86pt]{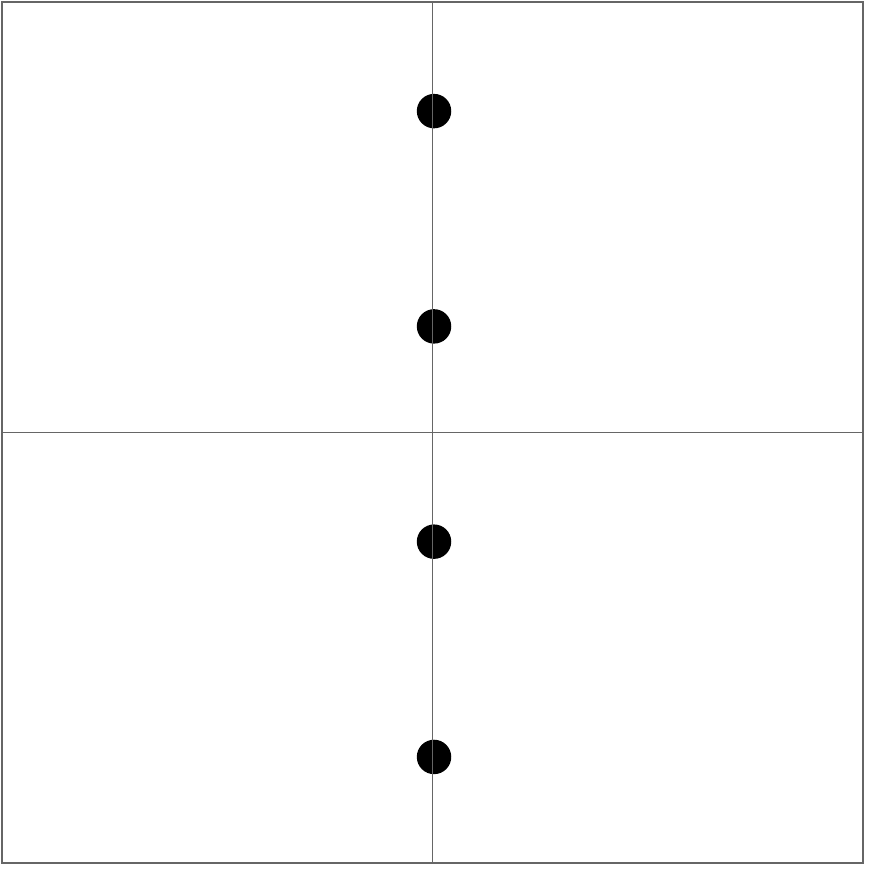}};
 \node at (8,12) {\includegraphics[width=86pt]{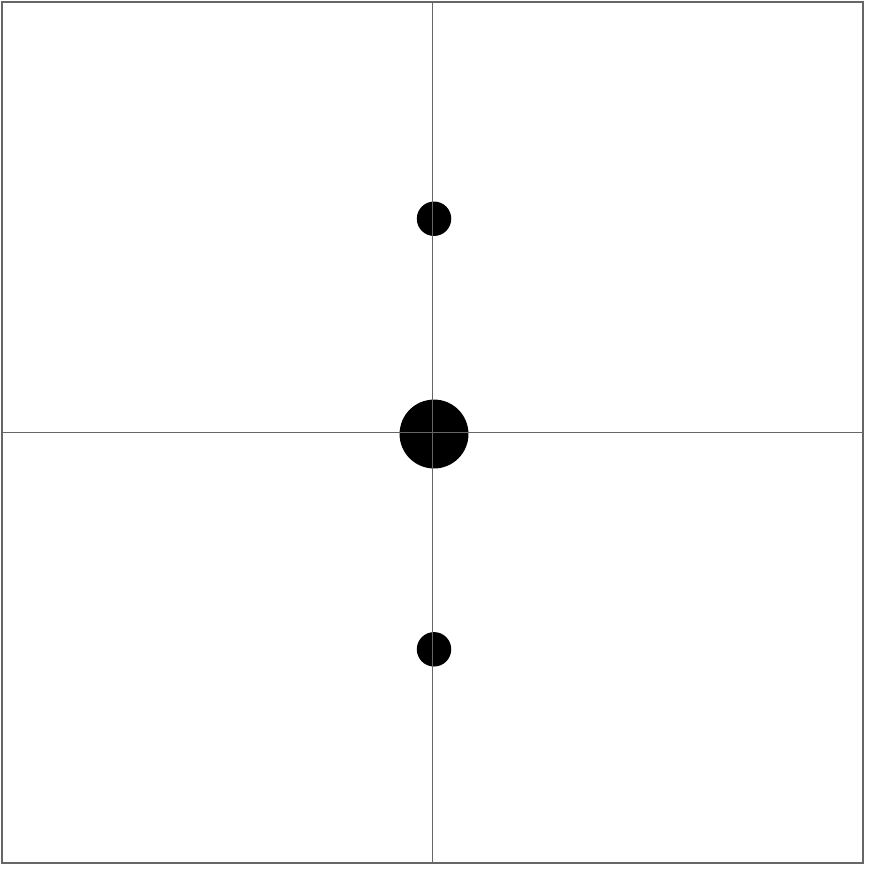}};
\end{tikzpicture}
\caption{The eigenvalue configurations on the $\mu,\nu,q_0$-hyperplane.}
\label{fig:hopf-fold-natural}
\end{figure}

The surface of equilibria is a 3-dimensional hypersurface so it cannot be drawn without a projection.
The Hopf surface is a 2-dimensional surface living on the surface of equilibria, which can be written in parametric form as
$$\left( \frac{25}{16}  \nu^2 q_0 -\frac{1}{2}q_0^2,\,-\nu^2-q_0,\,\nu,\,q_0\right) .$$
Its projection on the $(\mu,\nu,q_0)$-hyperplane is shown in Figure \ref{fig:fold-surf-nat} and its projection on the $(\kappa,\mu,\nu)$-hyperplane is shown in Figure \ref{fig:fold-surf-real}.

\paragraph{The fold surface} \hspace{0em}

Th fold surface can be written in parametric form as
$$\left( \frac{1}{2}q_0^2,\,\frac{9}{16} \nu^2-q_0,\,\nu,\,q_0 \right) .$$
Its projection on the $(\mu,\nu,q_0)$-hyperplane is shown in Figure \ref{fig:hopf-surf-nat} and its projection on the $(\kappa,\mu,\nu)$-hyperplane is shown in Figure \ref{fig:hopf-surf-real}.

\paragraph{The eigenvalue configuration} \hspace{0em}

Together the two surfaces reveal the eigenvalue configurations of the system. 
In Figure \ref{fig:hopf-fold-natural} their projection to the $(\mu,\nu,q_0)$-hyperplane is shown.
In this projection each point corresponds to exactly one equilibrium point.

The fold surface has different meanings in different projections.
Recall that each point on the $(\mu,\nu,q_0)$-hyperplane gets mapped to exactly one equilibrium.
On the other hand, in the $(\kappa,\mu,\nu)$ projection the fold surface separates the space into two regions, one where 2 equilibria exist and another where no equilibrium exists.
The eigenvalue configurations of this projection is shown in Figure \ref{fig:hopf-fold-real}.

\section{Parameter reduction}
\label{ch:param_redux}

Since $q_1^3$ is the only third order term in the Hamiltonian \eqref{eq:truncated-hamiltonian}, we can use it to cancel the term $q_1^2$.
This will reduce the parameters to two, as expected from the linear theory.

\subsection{The natural reduction}
\label{ch:natural-reduction}

Using the translation $q_1 \mapsto q_1-\mu/a_1$, $p_2 \mapsto p_2 + 3 \mu \nu/4 a_1$, time reparametri\-zation and possibly a sign change of $q_1$,  the Hamiltonian \eqref{eq:truncated-hamiltonian} gets transformed to
\begin{equation*}
  H_t = \frac{p_2^2}{2}-p_1 q_2 +\alpha q_1 + \beta \left( \frac{q_2^2}{2}+\frac{3}{4}p_2 q_1 \right)+ \frac{ q_1^3}{6},
\end{equation*}
with $\alpha=\kappa-\mu^2/2+9\mu\nu^2/16$ and $\beta=\nu$ being considered as the new parameters. The equations of motion take the form
\begin{align*}
\dot{q}_1&=-q_2, \\
\dot{q}_2&=p_2+\frac{3}{4} \beta q_1, \\
\dot{p}_1&=-\alpha-\frac{3}{4} \beta p_2-\frac{1}{2}q_1^2, \\
\dot{p}_2&=p_1-\beta q_2.
\end{align*}
An equilibrium of these equations has the form $(q_0,0,0,-\frac{3}{4} \beta q_0)$ with $q_0$ satisfying
\begin{equation}
\frac{1}{2} q_0^2-\frac{9}{16} \beta^2 q_0+\alpha=0,
\label{eq:surf-equil-param-reduct}
\end{equation}
which defines the surface of equilibria, a saddle surface in this case.

As usually, $q_0$ can be viewed as a new parameter and equation \eqref{eq:surf-equil-param-reduct} defines a smooth 2-dimensional surface in the 3-dimensional parameter space $(q_0,\alpha,\beta)$. We define a surjection from the $(q_0,\beta)$-plane to the set of equilibria by mapping the pair $(q_0,\beta)$ to the equilibrium $(q_0,0,0,-\frac{3}{4} \beta q_0)$. The eigenvalues of this equilibrium are $\pm\sqrt{-\frac{5}{4} \beta\pm\sqrt{\beta^2+q_0}}$.

\begin{figure}[p]
  \centering
  \subfloat[\label{fig:syst-example-surf-equil-1}]{\includegraphics[width=0.48\textwidth]{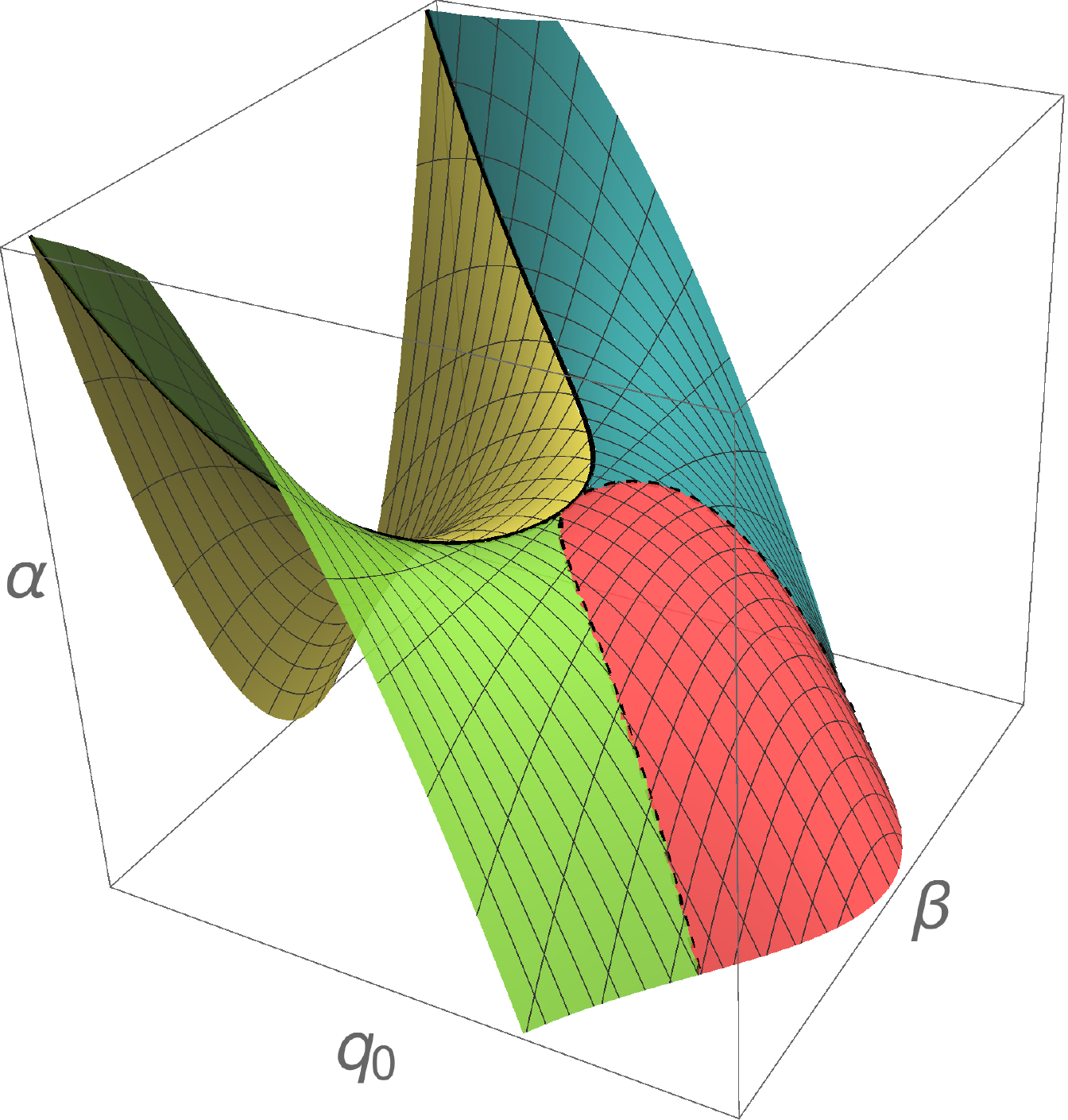}} \hspace{0.02\textwidth}               
  \subfloat[\label{fig:syst-example-surf-equil-2}]{\includegraphics[width=0.48\textwidth]{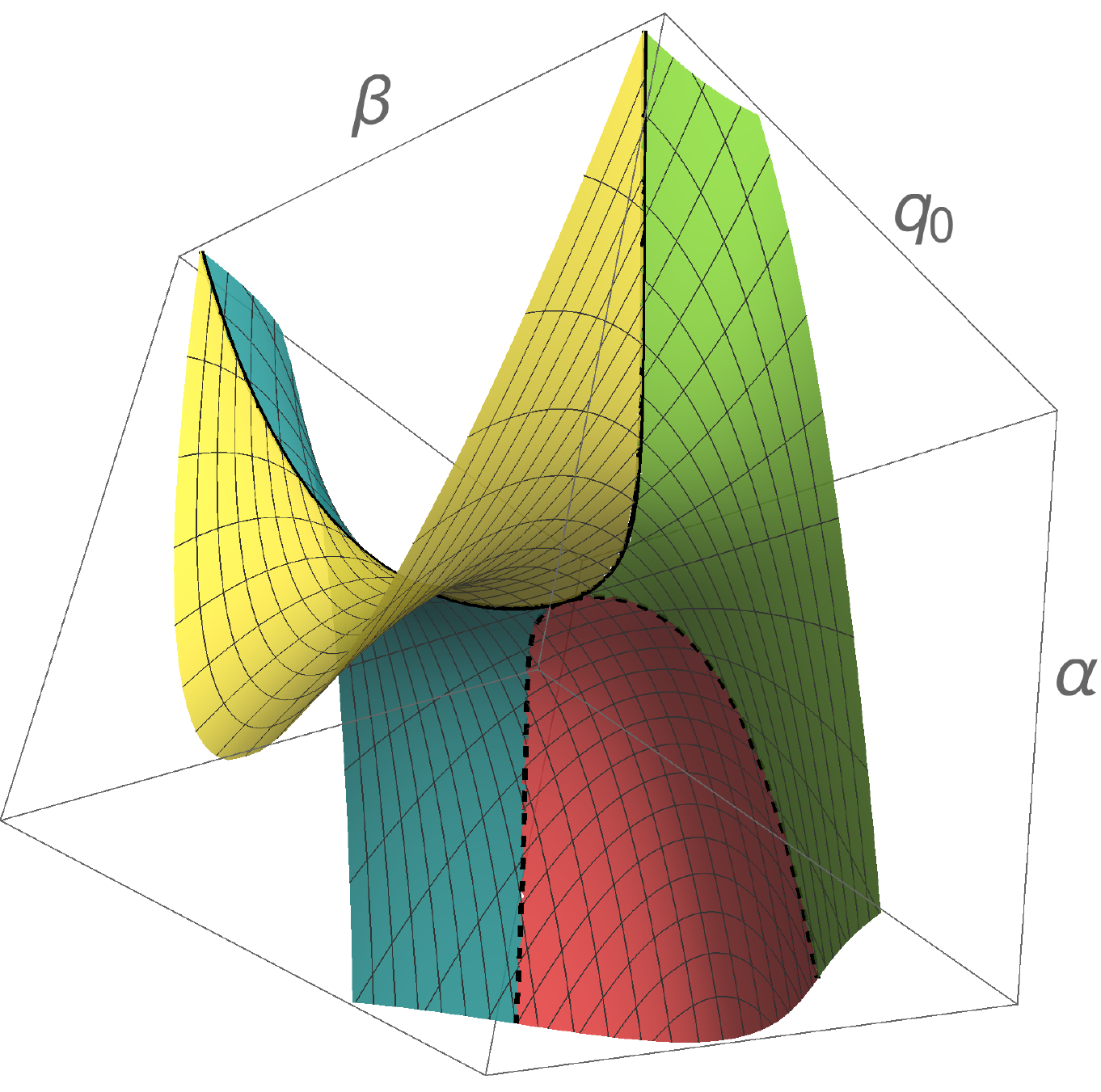}} \hspace{0.02\textwidth}
  \subfloat[\label{fig:syst-example-surf-equil-3}]{\includegraphics[width=0.48\textwidth]{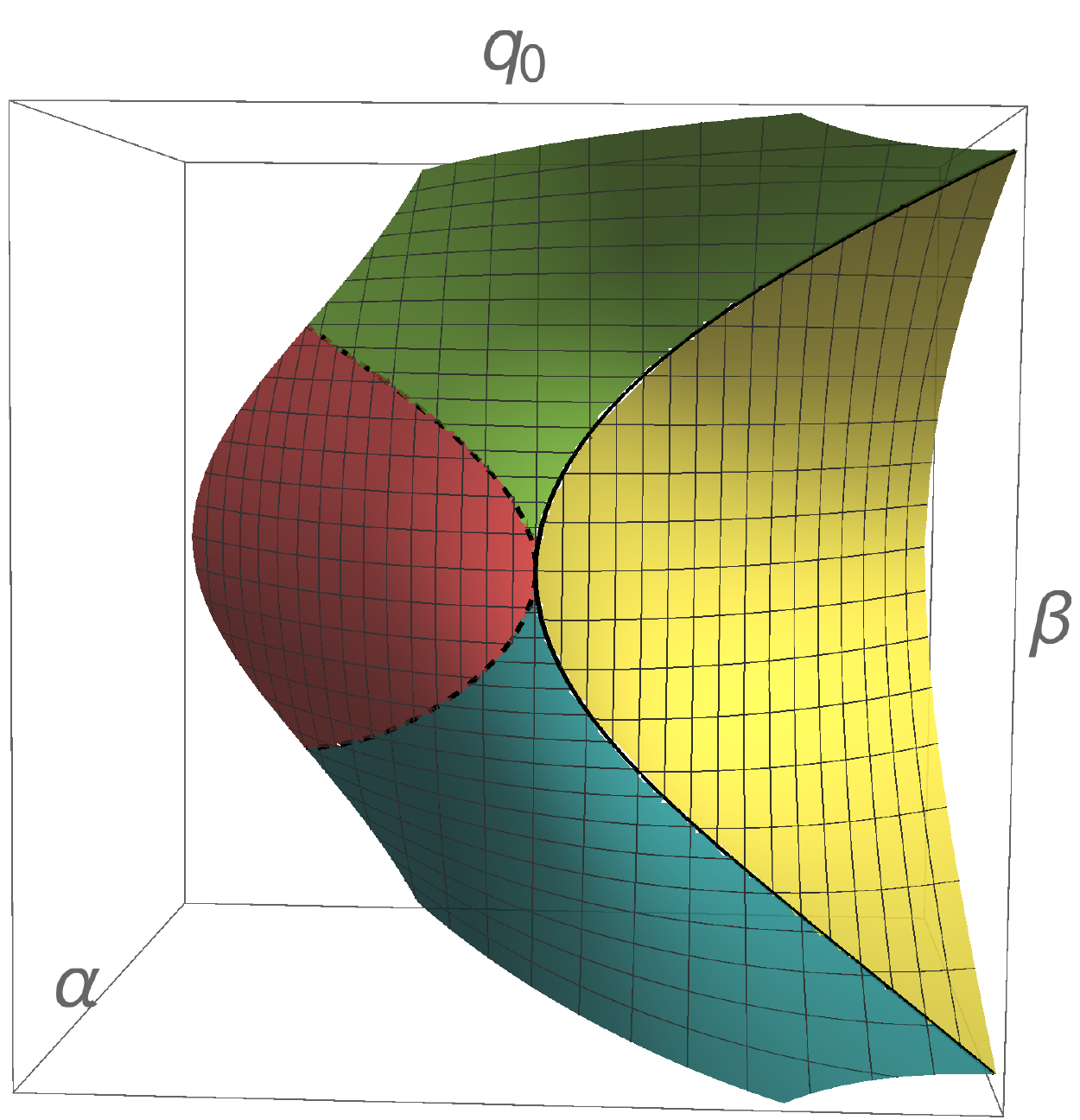}} \hspace{0.02\textwidth}
  \subfloat[\label{fig:syst-example-surf-equil-4}]{\includegraphics[width=0.48\textwidth]{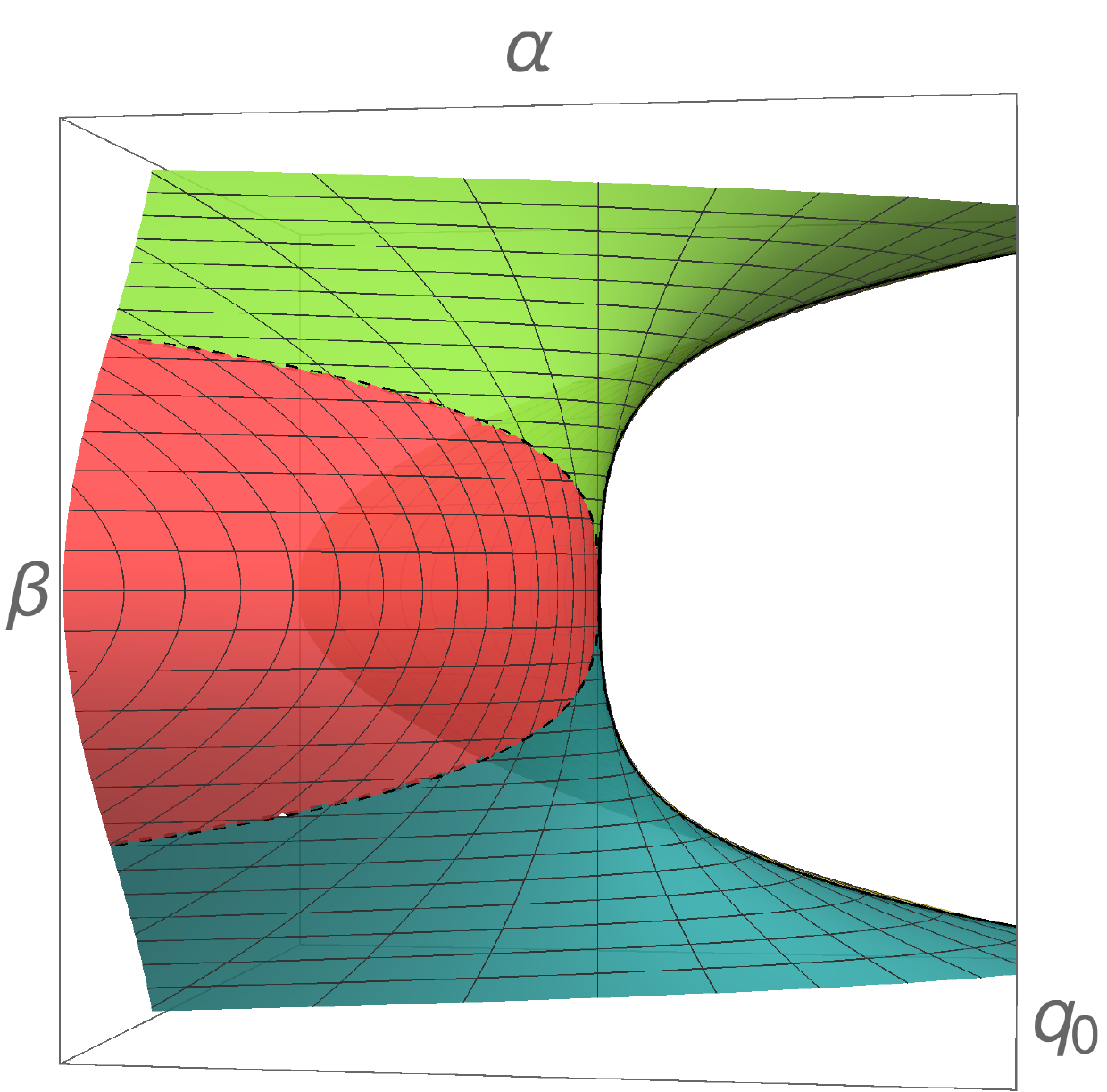}}
  \caption{The surface of equilibria viewed from different angles.}
  \label{fig:syst-example-surf-equil}
\end{figure}

\begin{figure}[p]
\begin{tikzpicture}
 \node at (6,6) {\includegraphics[width=220pt]{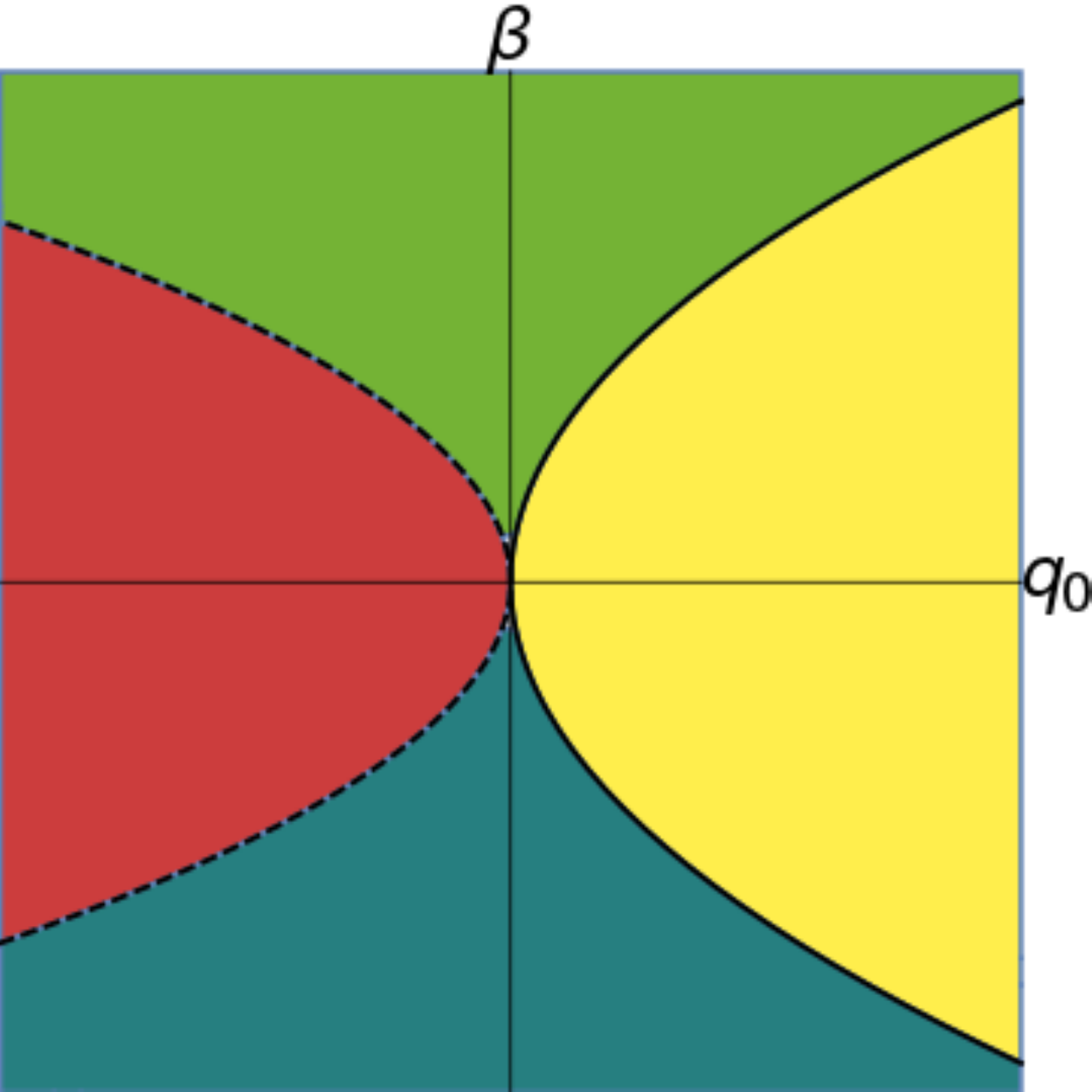}};
 \draw[->] (7.3,10.5) -- (8,8.48);
 \draw[->] (4.5,10.5) -- (5.5,8.48);
 \draw[->] (1.4,10.51) -- (2.7,8.13);
 \draw[->] (1.5,7.5) -- (3.5,6);
 \draw[->] (1.5,4.5) -- (3.18,3.6);
 \draw[->] (1.5,1.4) -- (5.3,3);
 \draw[->] (5.5,1.5) -- (6.7,3.95);
 \draw[->] (8.5,1.5) -- (8.4,5);
 \node at (0,12) {\includegraphics[width=85pt]{pics/dim.pdf}};
 \node at (0,8) {\includegraphics[width=86pt]{pics/cmplx.pdf}};
 \node at (0,4) {\includegraphics[width=86pt]{pics/dre.pdf}};
 \node at (0,0) {\includegraphics[width=86pt]{pics/re-re.pdf}};
 \node at (4,0) {\includegraphics[width=86pt]{pics/re-null.pdf}};
 \node at (8,0) {\includegraphics[width=86pt]{pics/re-im.pdf}};
 \node at (4,12) {\includegraphics[width=86pt]{pics/im-im.pdf}};
 \node at (8,12) {\includegraphics[width=86pt]{pics/im-null.pdf}};
\end{tikzpicture}
\caption{The eigenvalue configurations on the plane $(\beta,q_0)$.}
\label{fig:syst-example-natural}
\end{figure}

\begin{figure}[p]
\begin{tikzpicture}
 \node at (6,6) {\includegraphics[width=220pt]{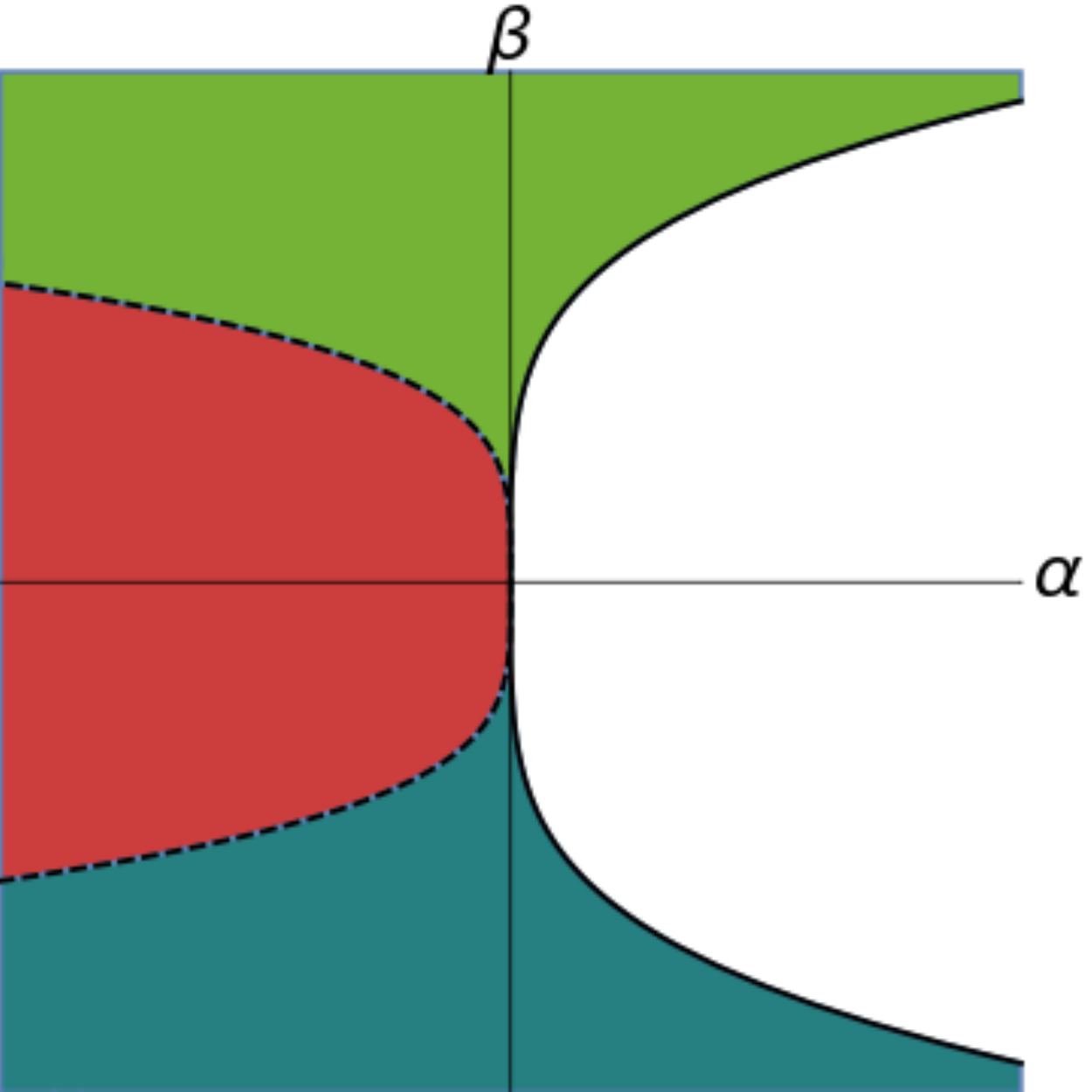}};
 \draw[->] (7.3,10.5) -- (8,8.8);
 \draw[->] (4.5,10.5) -- (5.5,8.48);
 \draw[->] (1.4,10.51) -- (2.7,7.8);
 \draw[->] (1.5,7.5) -- (2.5,7);
 \draw[->] (1.5,4.5) -- (2.5,3.7);
 \draw[->] (1.5,1.4) -- (5.3,3.5);
 \draw[->] (5.5,1.5) -- (6.8,3.2);
 \draw[->] (8.5,1.5) -- (8.4,5);
 \node at (0,12) {\includegraphics[width=85pt]{pics/dim2.pdf}};
 \node at (0,8) {\includegraphics[width=86pt]{pics/cmplx2.pdf}};
 \node at (0,4) {\includegraphics[width=86pt]{pics/dre2.pdf}};
 \node at (0,0) {\includegraphics[width=86pt]{pics/re-re2.pdf}};
 \node at (4,0) {\includegraphics[width=86pt]{pics/re-null2.pdf}};
 \node at (8,0) {\includegraphics[width=86pt]{pics/noeq.pdf}};
 \node at (4,12) {\includegraphics[width=86pt]{pics/im-im2.pdf}};
 \node at (8,12) {\includegraphics[width=86pt]{pics/im-null2.pdf}};
 \node at (8,0) {No equilibria};
\end{tikzpicture}
\caption{The eigenvalue configurations on the plane $(\alpha,\beta)$.}
\label{fig:syst-example-real}
\end{figure}

We know by the previous analysis that the Hamiltonian-Hopf bifurcation and the centre-saddle bifurcation always happen in this system.
Using the same techniques as before we find that the Hamiltonian-Hopf bifurcation happens on a line on the surface of equilibria, defined by $q_0=-\beta^2$, which is the Hopf ``surface''.
The fold ``surface'' can be found to be the line defined by $q_0=\frac{9}{16}\beta^2$.

Since in this case the surface of equilibria is two dimensional, it can be drawn without the need of projections.
This is done in Figures \ref{fig:syst-example-surf-equil}.
The different colours on the surface correspond to different eigenvalue configurations of the system.
The solid line is the fold ``surface'' and the dashed line is the Hopf ``surface''.

Figure \ref{fig:syst-example-surf-equil-3} corresponds to the projection of the surface on the $(\beta,q_0)$ plane.
The eigenvalue configurations of the system in this parametrization on this plane is shown is Figure \ref{fig:syst-example-natural}.
Notice that this is basically identical to the eigenvalue configurations of the unfolding of the linear system, shown in Figure \ref{fig:lin-eig-conf}.

In order to find the eigenvalue configurations of the system in the original parameters, we need to project the surface of equilibria on the plane $(\alpha,\beta)$.
Figure \ref{fig:syst-example-surf-equil-4} corresponds to this projection.
The eigenvalue configurations of the system is shown in Figure \ref{fig:syst-example-real}.
On this projection the fold ``surface'' is given by the equation $512 \alpha=81 \beta^4$ and the Hopf ``surface'' by the equation $16\alpha=-17 \beta^4$.
It should be noted that in \cite{hanssmann06} the two ``surfaces'' are wrongly depicted to have a 1st order tangency instead of 3rd.

\subsection{A more general view on reduction}

\begin{figure}[t]
  \centering
  \subfloat[$r<0$]{\label{fig:reduction-potraits-1}\includegraphics[width=0.3\textwidth]{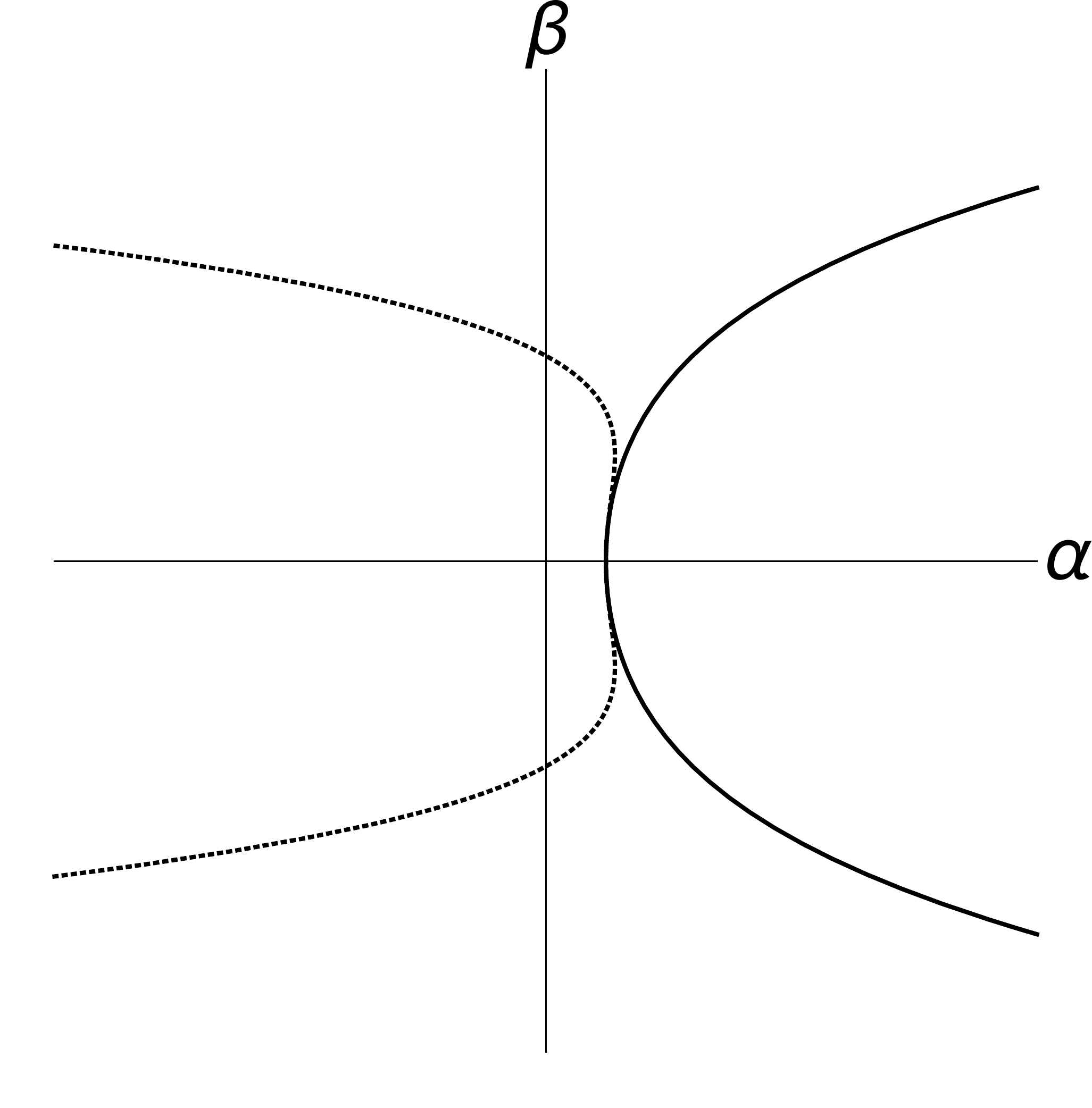}} \hspace{0.02\textwidth}
  \subfloat[$r=0$]{\label{fig:reduction-potraits-2}\includegraphics[width=0.3\textwidth]{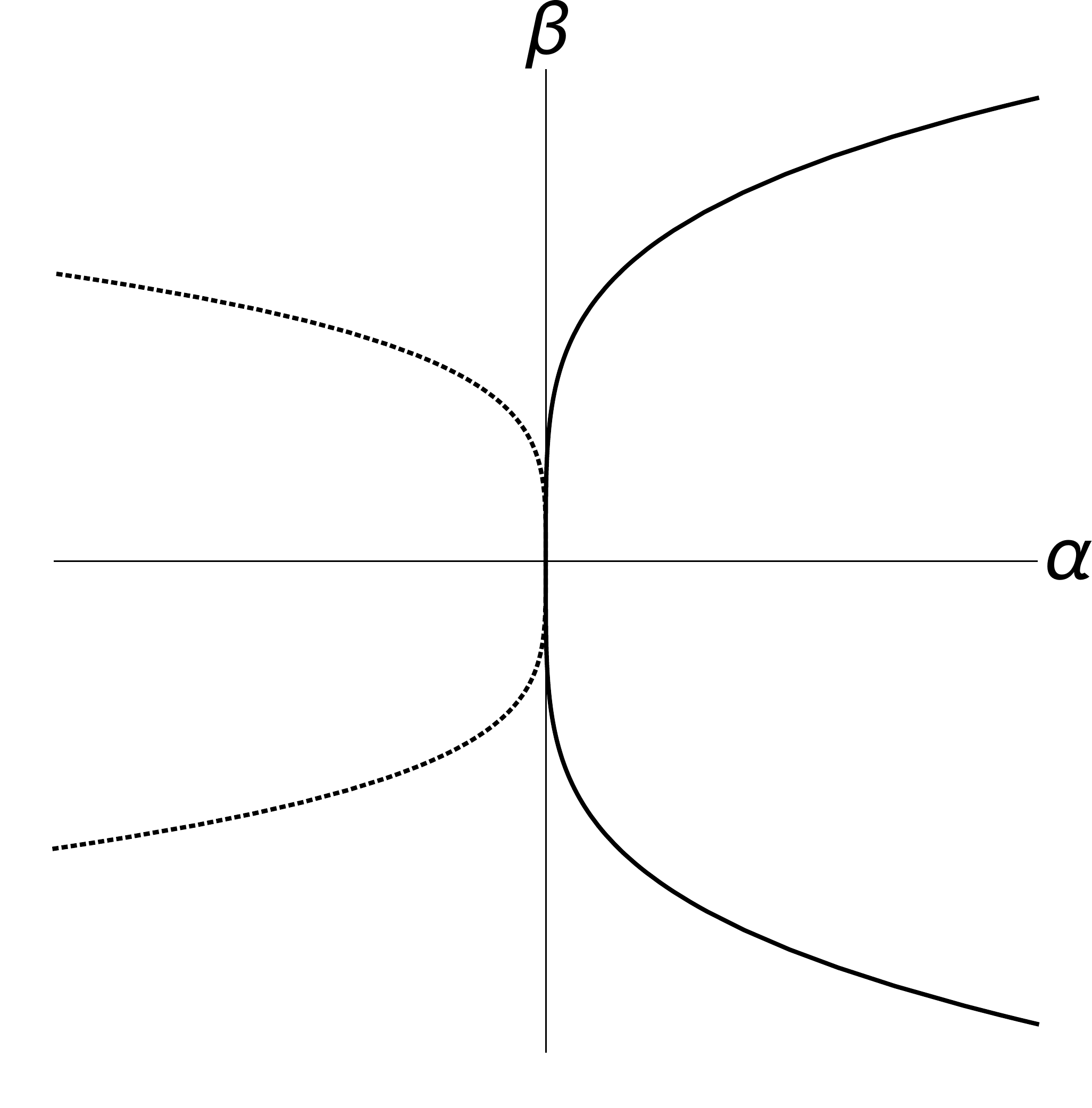}} \hspace{0.02\textwidth}
  \subfloat[$r>0$]{\label{fig:reduction-potraits-4}\includegraphics[width=0.3\textwidth]{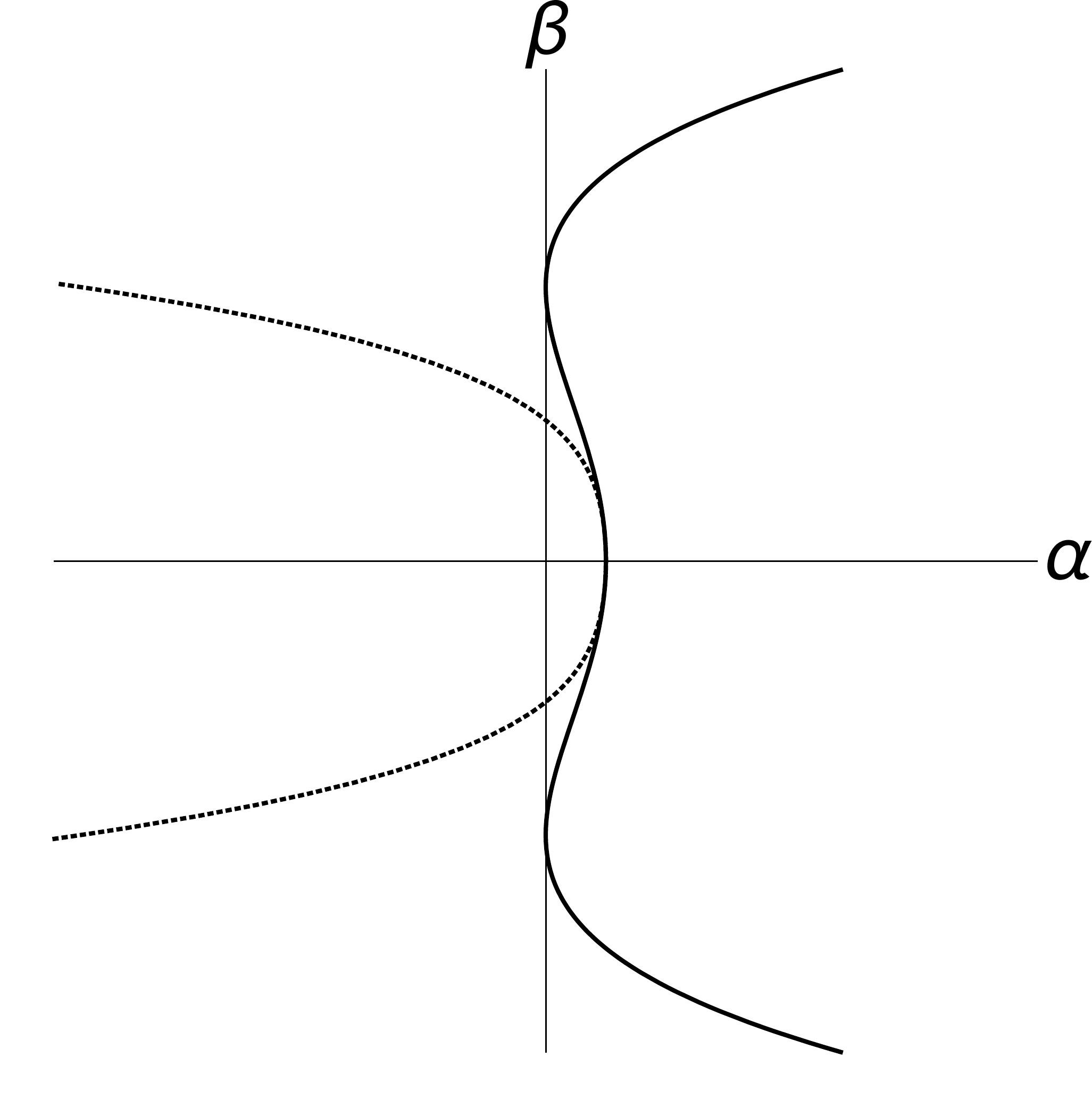}}
  \caption{Eigenvalue configurations for different values of $r$.}
  \label{fig:reduction-portraits}
\end{figure}

The fact that the two ``surfaces'' are defined by fourth degree equations is non-generic.
This non-genericity comes from the fact that the reduction that was performed in section \ref{ch:natural-reduction} was also non-generic.

Let us consider again the Hamiltonian
\begin{equation*}
 H = \frac{p_2^2}{2}-p_1 q_2 +\kappa q_1 +\mu \frac{ q_1^2}{2}+\nu\left( \frac{q_2^2}{2}+\frac{3}{4}p_2 q_1 \right)+a_1 \frac{ q_1^3}{6}.
\end{equation*}
Fix an $r\in\mathbb{R}$ small enough. Then the translation 
\begin{equation*}
q_1 \mapsto q_1+ \frac{r-\mu}{a_1},\;p_2 \mapsto p_2 +\frac{3 \nu( \mu-r)}{4 a_1} 
\end{equation*}
transforms the above Hamiltonian to
\begin{equation*}
 H_r = \frac{p_2^2}{2}-p_1 q_2 +\alpha q_1 +r \frac{ q_1^2}{2}+\beta\left( \frac{q_2^2}{2}+\frac{3}{4}p_2 q_1 \right)+a_1 \frac{ q_1^3}{6},
\end{equation*}
with $\beta=\nu$, $\alpha=\frac{r^2}{2 a_1}+\kappa-\frac{\mu^2}{2 a_1}-\frac{9 \nu^2 r}{16 a_1}+\frac{9 \mu \nu^2}{16 a_1}$. 
As above we can scale $a_1$ to 1.

An equilibrium of the system is of the form $(q_0,0,0,-\frac{3}{4} \beta q_0)$, but now $q_0$ satisfies
\begin{equation*}
\frac{1}{2} q_0^2-\frac{9}{16} \beta^2 q_0+r q_0+\alpha=0
\end{equation*}
and its eigenvalues are $\pm\sqrt{-\frac{5}{4} \beta\pm\sqrt{r+\beta^2+q_0}}.$

Now it is clear that $r$ can be set to any (small) number and all these seemingly different systems are equivalent. The reduction in section \ref{ch:natural-reduction} is the obvious one, hence the name natural, but it is only one from a continuum of possibilities.

In this system the fold ``surface'' is given by the equation $q_0=-r +\frac{9}{16}\beta^2$ and the Hopf ``surface'' by the equation $q_0=-r-\beta^2$.
Since $\alpha=\frac{9}{16} \beta^2-\frac{1}{2} q_0^2-r q_0$, we find
$$\alpha_f(\beta)=\frac{1}{2} r^2-\frac{9}{16} r \beta^2-\frac{17}{16} \beta^4 \;\;\mbox{ and }\;\; \alpha_h(\beta)=\frac{1}{2} r^2-\frac{9}{16} r \beta^2+\frac{81}{512} \beta^4.$$

Since the functions $\alpha_f(\beta)$ and $\alpha_h(\beta)$ are even and share the same constant and second order terms, they are tangent to degree 3 for every $r$.
The reduction in section \ref{ch:natural-reduction} is special in the sense that it makes this second order term disappear.
The eigenvalue configurations of the system for various values of $r$ is shown in Figures \ref{fig:reduction-portraits}.
Notice that by choosing different values for $r$ one takes different cuts of constant $\mu$ of the surfaces in Figure \ref{fig:hopf-fold-real}.

\section*{Acknowledgement}
The author thanks Heinz Han{\ss}mann for introducing the problem and for many fruitful discussions.

\bibliographystyle{alpha}
\bibliography{NilpotEqpaper}

\def\cprime{$'$}
\begin{thebibliography}{BCKV93}

\bibitem[AKN06]{arnold06}
V.~I. Arnol{\cprime}d, V.~V. Kozlov, and A.~I. Neishtadt.
\newblock {\em Mathematical aspects of classical and celestial mechanics},
  volume~3 of {\em Encyclopaedia of Mathematical Sciences}.
\newblock Springer-Verlag, Berlin, third edition, 2006.
\newblock [Dynamical systems. III], Translated from the Russian original by E.
  Khukhro.

\bibitem[Arn71]{arnold71}
V.~I. Arnol{\cprime}d.
\newblock Matrices depending on parameters.
\newblock {\em Russ Math Surv}, 26(2):29--43, 1971.

\bibitem[Arn90]{arnold90}
Vladimir~I. Arnol{\cprime}d.
\newblock {\em Mathematical methods of classical mechanics}, volume~60 of {\em
  Graduate Texts in Mathematics}.
\newblock Springer-Verlag, New York, 1990.
\newblock Translated from the 1974 Russian original by K. Vogtmann and A.
  Weinstein, Corrected reprint of the second (1989) edition.

\bibitem[BCKV93]{MR1227332}
H.~W. Broer, S.-N. Chow, Y.~Kim, and G.~Vegter.
\newblock A normally elliptic {H}amiltonian bifurcation.
\newblock {\em Z. Angew. Math. Phys.}, 44(3):389--432, 1993.

\bibitem[BV10]{basov10}
V.~V. Basov and A.~S. Vaganyan.
\newblock Normal forms of {H}amiltonian systems.
\newblock {\em Differ. Uravn. Protsessy Upr.}, (4):86--107, 2010.

\bibitem[Car81]{Carr06}
Jack Carr.
\newblock {\em Applications of centre manifold theory}, volume~35 of {\em
  Applied Mathematical Sciences}.
\newblock Springer-Verlag, New York, 1981.

\bibitem[CS87]{cushman-sanders86}
R.~Cushman and J.~A. Sanders.
\newblock Invariant theory and normal form of {H}amiltonian vectorfields with
  nilpotent linear part.
\newblock In {\em Oscillations, bifurcation and chaos ({T}oronto, {O}nt.,
  1986)}, volume~8 of {\em CMS Conf. Proc.}, pages 353--371. Amer. Math. Soc.,
  Providence, RI, 1987.

\bibitem[Gal82]{galin82}
D.~M. Galin.
\newblock Versal deformations of linear {H}amiltonian systems.
\newblock {\em AMS Transl.}, (118):1--12, 1982.

\bibitem[Han07]{hanssmann06}
H.~Han{\ss}mann.
\newblock {\em Local and semi-local bifurcations in {H}amiltonian dynamical
  systems: Results and examples}, volume 1893 of {\em Lecture Notes in
  Mathematics}.
\newblock Springer-Verlag, Berlin, 2007.

\bibitem[Hum78]{humphreys78}
J.~E. Humphreys.
\newblock {\em Introduction to {L}ie algebras and representation theory},
  volume~9 of {\em Graduate Texts in Mathematics}.
\newblock Springer-Verlag, New York, 1978.
\newblock Second printing, revised.

\bibitem[vdM82]{vdMeer82}
J.C. van~der Meer.
\newblock Nonsemisimple {$1:1$} resonance at an equilibrium.
\newblock {\em Celestial Mech.}, 27(2):131--149, 1982.

\bibitem[vdM85]{vdMeer85}
J.C. van~der Meer.
\newblock {\em The {H}amiltonian {H}opf bifurcation}, volume 1160 of {\em
  Lecture Notes in Mathematics}.
\newblock Springer-Verlag, Berlin, 1985.

\bibitem[vdM86]{vdMeer86}
J.C. van~der Meer.
\newblock Bifurcation at nonsemisimple {$1\colon -1$} resonance.
\newblock {\em Z. Angew. Math. Phys.}, 37(3):425--437, 1986.

\bibitem[Wag02]{Wagen02}
Thomas Wagenknecht.
\newblock Bifurcation of a reversible {H}amiltonian system from a fixed point
  with fourfold eigenvalue zero.
\newblock {\em Dyn. Syst.}, 17(1):29--44, 2002.

\bibitem[Wil36]{williamson36}
J.~Williamson.
\newblock On the {A}lgebraic {P}roblem {C}oncerning the {N}ormal {F}orms of
  {L}inear {D}ynamical {S}ystems.
\newblock {\em Amer. J. Math.}, 58(1):141--163, 1936.

\end{thebibliography}

\end{document}